\documentclass[a4paper,10pt,pdflatex,reqno]{amsart}
\usepackage{amsmath,times} \usepackage{amsthm}

\usepackage{graphicx}
\usepackage{color}

\usepackage{amssymb} \usepackage{latexsym} \usepackage{amscd}\usepackage{hyperref}
\usepackage{verbatim} 
\usepackage[latin1]{inputenc} 
\newtheorem{theorem}{Theorem}[section]
\newtheorem{lemma}[theorem]{Lemma}
\newtheorem{proposition}[theorem]{Proposition}

\usepackage{lineno}

\theoremstyle{definition} \newtheorem{definition}[theorem]{Definition}

\newtheorem{remark}[theorem]{Remark}

\begin{document}

\title[HMAE and Canonical Tubular Neighbourhoods in K\"ahler Geometry]{Homogeneous Monge-Amp\`ere Equations and Canonical Tubular Neighbourhoods in K\"ahler Geometry}

\author{Julius  Ross and David Witt Nystr\"om}
\maketitle

\begin{abstract}
  We prove the existence of canonical tubular neighbourhoods around complex submanifolds of K\"ahler manifolds that are adapted to both the holomorphic and symplectic structure. This is done by solving the complex Homogeneous Monge-Amp\`ere equation on the deformation to the normal cone of the submanifold. We use this to establish local regularity for global weak solutions, giving local smoothness to the (weak) geodesic ray in the space of (weak) K\"ahler potentials associated to a given complex submanifold. We also use it to get an optimal regularity result for naturally defined plurisubharmonic envelopes and for the boundaries of their associated equilibrium sets. 
\end{abstract}

\newcommand{\dist}{\operatorname{dist}}
\newcommand{\PSH}{PSH} 
\newcommand{\Sing}{\operatorname{Sing}}
\newcommand{\vol}{\operatorname{vol}}
\newcommand{\id}{\operatorname{id}}
\newcommand{\Conv}{\operatorname{Conv}}
\newcommand{\ord}{\operatorname{ord}}
\newcommand{\im}{\operatorname{im}}
\makeatletter
\newcommand{\Vast}{\bBigg@{5}}
\makeatother

\newcommand{\customV}{\Vast\}V}
\setlength{\parskip}{2pt}
\setlength{\parindent}{0pt}

\section{Introduction}

\subsection{Canonical tubular neighbourhoods}

Tubular neighbourhoods are used in differential and symplectic geometry to reveal the structure around submanifolds. As is well known, a complex submanifold of a complex manifold will in general not admit a tubular neighbourhood that is holomorphic. However, we will see in this paper that any compact complex submanifold of a K\"ahler manifold has a canonical smooth tubular neighbourhood, which in general will not be holomorphic, but nevertheless has properties that make it adapted to both the holomorphic and symplectic structure.   

To state our results, let $p\colon N_Y\to Y$ be the normal bundle of a complex submanifold $Y$ of a complex manifold $X$ and $\iota\colon Y\to N_Y$ be the inclusion of $Y$ as the zero section. Note that $N_Y$ is a holomorphic vector bundle that admits a holomorphic $S^1$-action obtained by rotating the fibres. By a smooth tubular neighbourhood of $Y\subset X$ we mean a diffeomorphism $T\colon U\to \tilde{U} \subset X$ between a neighbourhood $U$ of $\iota(Y)\subset N_Y$ and a neighbourhood $\tilde{U}$ of $Y\subset X$ such that $$T\circ \iota = \id_Y.$$

\begin{theorem}\label{thm:tubularmain}
Suppose that $Y$ is a complex submanifold of a K\"ahler manifold $(X,\omega)$.    Then there exists a smooth tubular neighbourhood $$T\colon U\to \tilde{U} \subset X$$ of $Y\subset X$ with the following properties:
\begin{enumerate}
\item $U$ is $S^1$-invariant and the pullback $T^*\omega$ is an $S^1$-invariant K\"ahler form on $U$.
\item For any $u\in U$ the function $f_u: S^1 \to X,$ given by $f_u(e^{i\theta}):=T(e^{i\theta}u)$ extends to a holomorphic function $F_u\colon D \to X$ from the unit disc $D\subset \mathbb C$ such that 
$$F_u(0) =p(u) \text{ and } \left[{DF_u}|_0(\frac{\partial}{\partial x}) \right]= u,$$ i.e. the holomorphic disc $F_u$ is centered at $p(u)$ and points in the normal direction determined by $u$. 
\end{enumerate}
When $Y$ is compact there is a canonical choice of $T$. In general, the germ of the tubular neighbourhood we construct is local and canonical, in the sense that at any point $\iota(p)\in \iota(Y)$ it depends only on the local structure of $(Y,\omega_{|Y})\subset (X,\omega)$ around $p\in Y$. 
\end{theorem}

Of course the map $T$ need not be holomorphic, but the two properties above describe ways in which it interacts with the holomorphic structure (for example the statement that $T^*\omega$ is K\"ahler is not immediate given that $T$ is not holomorphic).  Even in cases where holomorphic tubular neighbourhoods do exist they will not in general have this property so what is produced above will be different. In fact the existence of such a tubular neighbourhood is highly non-trivial even when the submanifold is a single point in $\mathbb{C}$. 

In the very special case that there exists a holomorphic $S^1$-action on $X$ that fixes $Y$ pointwise inducing the usual action on $N_Y$ obtained by rotating the fibers there is a natural holomorphic tubular neighbourhood of $Y$ which encodes this symmetry. If in addition this action preserves $\omega$ then the germ of our $T$ will agree with this holomorphic tubular neighbourhood. Since our construction is local this is also true locally around any point $p\in Y$ (see Proposition \ref{prop:hologerm}).

We point out that there is a standard way to produce tubular neighbourhoods that only satisfy the first property in the above theorem.   In fact, for any choice of symplectic form $\tilde{\omega}$  on a neighbourhood $U$ of $\iota(Y)$ such that $\iota^*\tilde{\omega} = \omega_{|Y}$ there exists, after shrinking $U$ if necessary, a tubular neighbourhood $T\colon U\to X$ such that $T^*\omega = \tilde{\omega}$ (this follows by using an arbitrary tubular neighbourhood to pull back $\omega$ and then applying the relative version of Moser's Theorem \cite[Theorem 7.4]{deSilva}).  So by choosing such an $\tilde{\omega}$ that is K\"ahler and $S^1$-invariant we get a tubular neighbourhood with property (1). This will of course depend on several choices, and from the point of view of K\"ahler geometry it is not clear why any particular choice is more natural than any other. 

On the other hand it is easy to construct a tubular neighbourhood with property (2) from Theorem \ref{thm:tubularmain}, starting with any smooth foliation by holomorphic curves of some neighbourhood of the exceptional divisor in the blowup of $X$ along $Y$, which is transverse to the exceptional divisor.   But in general one has then lost control over the symplectic structure. Thus the significance of the above theorem is that there exist tubular neighbourhoods with both properties simultaneously, which as we shall see can be made canonical.

\subsection{HMAE on the deformation to the normal cone of Y}

Really the main result of this paper is a proof of existence of regular solutions to a Dirichlet problem for a certain homogeneous Monge-Amp\`ere equation (HMAE), and the tubular neighbourhood of Theorem \ref{thm:tubularmain} will then be constructed using the associated foliation.  

Recall the \emph{deformation to the normal cone} $\mathcal N_Y$ of $Y$ in $X$ is the blowup of $X\times D$ along $Y \times \{0\}$, with projection map
$$\pi\colon \mathcal N_Y\to X\times D.$$ 
We let $\pi_X\colon \mathcal N_Y\to X$  and $\pi_D\colon \mathcal N_Y\to D$ denote the composition of $\pi$ with the projection onto the factors $X$ and $D$ of $X\times D$ respectively.  The fiber $\pi_D^{-1}(0)$ of $\mathcal N_Y$ over $0\in D$ has two components, one which is isomorphic to the blowup $Bl_Y(X)$ of $X$ along $Y$ and the other being the exceptional divisor $E$ which is isomorphic to the projective completion $\mathbb P(N_Y\oplus\mathbb C)$ of the normal bundle $N_Y$.   Since $N_Y$ embeds in $\mathbb P(N_Y\oplus\Bbb C)$ by $v\mapsto [v,1]$ we can, and shall,  identify $N_Y$ with its image in $\mathcal N_Y$.  Thus $\mathcal N_Y$ provides us with a space which contains both our original manifold $X$ as well as the normal bundle $N_Y$ (sitting inside its projective completion), and we will use this to construct a tubular neighbourhood as a kind of flow given by a foliation associated to a solution to a certain Dirichlet problem that we describe next.

  Letting $\mathcal Y\subset \mathcal N_Y$ denote the proper transform of $Y\times D$  we have that $\iota(Y)= E\cap \mathcal Y$.  Similarly if $p\in Y$ we let $\mathcal D_p$ denote the proper transform of $\{p\}\times D$, so $\iota(p) = E\cap \mathcal D_p$.    The holomorphic $S^1$-action on $X\times D$ given by 
\begin{equation} \label{circleaction}
e^{i\theta}\cdot (x,\tau) =(x,e^{-i\theta}\tau)
\end{equation} 
lifts to $\mathcal N_Y$; it preserves $N_Y$ and induces the action on $N_Y$ obtained by rotating the fibers.

\begin{theorem} \label{defconethm}
There exists an $S^1$-invariant neighbourhood $V$ of $\mathcal Y$ in $\mathcal N_Y$ and $S^1$-invariant regular solution $\Omega$ to the homogeneous Monge-Amp\`ere equation with boundary data induced by $\omega$.  Here, by a regular solution we mean that:
\begin{enumerate}
\item $\Omega$ is a smooth real closed  $(1,1)$ form on $V$,
\item $\Omega_{|V_{\tau}}$ is K\"ahler for all $\tau\in D$ where $V_{\tau}:= \pi^{-1}_D(\tau)\cap V$,
\item $\Omega_{|V_\tau}=\pi_X^*\omega_{|V_\tau}$ for all $\tau\in S^1$, and
\item $\Omega^{n+1}=0$ on $V$.
\end{enumerate}
The solution $\Omega$ is cohomologous to $\pi_X^*\omega,$ i.e. there exists a (unique) smooth real valued $S^1$-invariant function $\Phi$ on $V$ which is zero on $V_\tau$ for $\tau\in S^1$ and such that $$\Omega=\pi_X^*\omega|_V+dd^c\Phi.$$
When $Y$ is compact there is a canonical choice of $(V,\Omega)$. In general the germ around $\mathcal Y$ of any such $S^1$-invariant regular solution is unique.
\end{theorem}

In the above theorem $V$ should be understood as not meeting the singular locus of the fibre of $\pi_D^{-1}(0)$ and hence $V$ as well as $V_{\tau}$ for $\tau\in D$ are smooth.  We observe that since the central fibre of $\mathcal N_Y$ is of a different topological type to the general fibre, there can never be a regular solution defined on all of $\mathcal N_Y$, so in this sense a local result such as Theorem \ref{defconethm} is the best one could hope for. The uniqueness part of Theorem \ref{defconethm} is also novel since it does not involve any hypothesis on the behaviour of the solutions near the boundary of $V$.  


We explain briefly how such an $\Omega$ gives rise to a tubular neighbourhood.  Since the work of Bedford-Kalka \cite{BedfordKalka} it has been known that regular solutions to the HMAE generate associated \emph{Monge-Amp\`ere foliations} by holomorphic curves as follows: the kernel of the form $\Omega$ defines an  integrable distribution of complex lines in the tangent bundle of $V$.  By Frobenius' theorem it induces a foliation by holomorphic curves, which because of property (1) are transverse to the fibers $V_{\tau}$. In the case under consideration, $\Omega$ is $S^1$-invariant and hence so is the foliation (see Figure 1). 

\begin{figure}[htb]
	\centering
\scalebox{.7}{\input{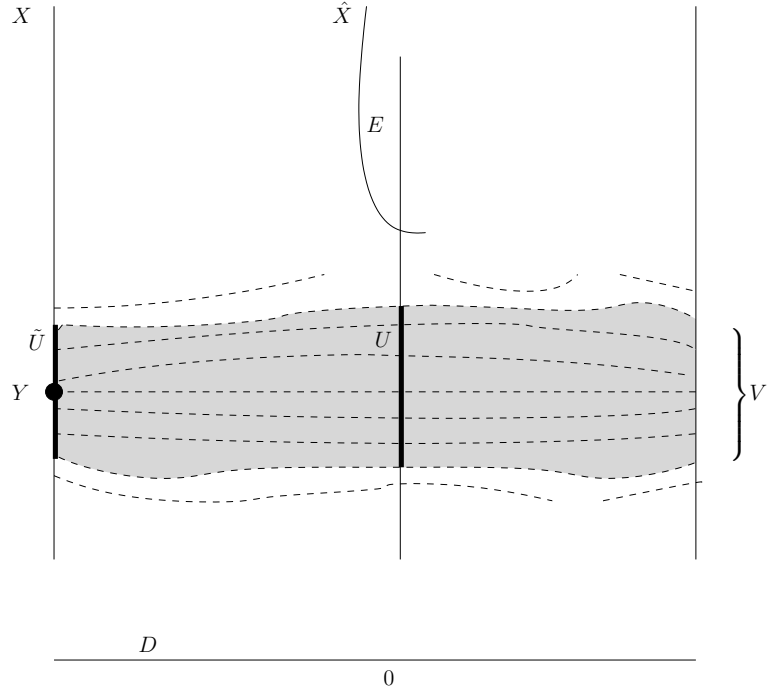}}
	\caption{Local regularity of HMAE in neighbourhood of the proper transform of $Y\times D$.  The dotted lines represent the leaves of the Monge-Amp\'ere foliation associated to a regular solution $\Omega$ of the HMAE.    The tubular neighbourhood $T:U\to\tilde{U}$ is obtained as the flow along these leaves.}
	\label{fig:blowup}
\end{figure}

We shall say that a leaf of the foliation is \emph{complete} if it covers the base $D$ (i.e.\ the restriction to this leaf of the projection to $D$ is a surjection onto $D$). From the $S^1$-invariance it follows that $\Omega|_{\mathcal D_p}=0$ for all $p\in Y$, thus each $\mathcal D_p$ is a complete leaf of the foliation. By continuity we can then find an $S^1$-invariant neighbourhood of $\mathcal Y$ consisting solely of complete leaves that pass through $N_Y\subset E$, so after shrinking $V$ we can assume that $V$ is foliated by such leaves. Note that $U:=V_0= V\cap \pi_D^{-1}(0)$ is an $S^1$-invariant neighbourhood of $\iota(Y)$ in $N_Y$.   By flowing along the leaves of this foliation we get for each $\tau\in D$ a diffeomorphism $T_{\tau}:U \to V_{\tau}$ which remarkably is a symplectomorphism, i.e. 
\begin{equation} \label{eq:pullbacknew}
T^*_{\tau}(\Omega_{|V_{\tau}})=\Omega_{|V_0}.
\end{equation}
(see \cite[Lemma 1]{Donaldson} or \cite{Semmes}). The map $T_1\colon U\to V_1$ is the desired tubular neighbourhood, and thanks to (\ref{eq:pullbacknew}) it has property (1) of Theorem \ref{thm:tubularmain}, while the existence of the holomorphic leaves assures that it has property (2) (the details are provided in Section \ref{sec:tubular}).

We now turn to making our uniqueness statement more precise.

\begin{definition}
We say that a solution $(V,\Omega)$ to the HMAE is \emph{complete} if $V$ is foliated by complete leaves of the Monge-Amp\`ere foliation, $V_0\subseteq N_Y$ and whenever $u\in V_0$ then $\tau u\in V_0$ for all $\tau \in D$.
\end{definition}

Given a regular solution $(V,\Omega)$ as in Theorem \ref{defconethm} we get a complete regular solution by shrinking $V$ if necessary.  

Let $\zeta$ denote the vector field which generates the $S^1$-action on $V$. From the fact that $\Omega$ is cohomologous to $\pi_X^*\omega$ it follows that the $S^1$-action is Hamiltonian, in the sense described in the next theorem. 

\begin{theorem}\label{thm:ham}
Let $(V,\Omega)$ be a regular solution to the HMAE as in Theorem \ref{defconethm}. Then the function $H:=L_{J\zeta}\Phi$ is a Hamiltonian for the $S^1$-action, in that it satisfies $$dH = \iota_{\zeta} \Omega.$$ Moreover $H$ is constant along the leaves of the Monge-Amp\`ere foliation associated to $\Omega$. If $(V,\Omega)$ is complete then $H\ge 0$ with equality precisely on $\mathcal Y$.
\end{theorem}

At least when $Y$ is compact we can use this Hamiltonian to select a canonical regular solution $(V_{can},\Omega_{can})$ to the HMAE.

\begin{definition}
Given a complete solution $(V,\Omega)$ with Hamiltonian $H$ we define the radius $rad(V,\Omega)$ of $(V,\Omega)$ to be the supremum of all $\lambda\geq 0$ such that $\partial H^{-1}([0,\lambda))\subseteq V$. We then define the \emph{canonical radius} $\Lambda_{can}$ to be the supremum of $rad(V,\Omega)$ over all complete solutions $(V,\Omega)$. 
\end{definition}

Note that the radius of a complete solution could be zero, and hence the canonical radius $\Lambda_{can}$ could also be zero. But at least when $Y$ is compact we clearly have that $rad(V,\Omega)>0,$ and thus $\Lambda_{can}>0$. 

\begin{theorem} \label{thm:canonicalhmae}
Assume that $Y$ is compact (or more generally $\Lambda_{can}>0$). Then there exists a unique complete solution $(V_{can},\Omega_{can})$ to the HMAE as in Theorem \ref{defconethm} with $$rad(V_{can},\Omega_{can})=\Lambda_{can}$$ and $$H_{can}< \Lambda_{can}.$$ This canonical solution is maximal in the following sense. If $(V,\Omega)$ is any other complete solution to the HMAE as in Theorem \ref{defconethm} then for any $\lambda<rad(V,\Omega)$ we have that $$H^{-1}([0,\lambda))=H^{-1}_{can}([0,\lambda))$$ and there $\Omega=\Omega_{can}$.
\end{theorem}

Thus when $\Lambda_{can}>0$ the function ${H_{can}}_{|_{V_1}}$ gives a canonical smooth function defined on a neighbourhood of $Y$ in $X$ (that depends on the K\"ahler structure of $X$ near $Y$), whose level sets give a canonical "flow" away from $Y$, which we will revisit below.

\subsection{Local regularity of weak solutions}

For topological reasons one cannot find global regular solutions of the HMAE on $\mathcal{N_Y}$ but instead of looking for local regular solutions one can consider global weak solutions. Using a variant of the Perron envelope, for any fixed number $c$ one can construct a closed positive $(1,1)$-current $\Omega_w$ on $\mathcal{N_Y}$ cohomologous to $\pi_X^*\omega-c[E]$ which restricts to $\omega$ on $X_{\tau}$ for $\tau\in S^1,$ and such that $$\Omega_w^{n+1}=0$$ in the sense of Bedford-Taylor. Thus $\Omega_w$ is a globally defined weak solution to the HMAE.

\begin{theorem}\label{thm:mainhmae2}
Suppose $Y$ is compact (or more generally $\Lambda_{can}>0$). Then the weak solution $\Omega_w$ is equal to $\Omega_{can}$ on $H_{can}^{-1}([0,c))$. In particular $\Omega_w$ is smooth and regular on a neighbourhood of $\mathcal{Y}$. 
\end{theorem}

In particular this gives a local regularity result for the weak geodesic rays that are naturally associated to the deformation to the normal cone (see Section \ref{sec:previous}).

\subsection{Volume Growth of Canonical Tubular Neighbourhoods}

Any Riemannian metric $g$ defines a tubular neighbourhood of $Y$ through the exponential map, whose boundary is of course the geodesic distance function $dist(\cdot,Y)$  to $Y$.  For our canonical tubular neighbourhoods, the analog of the function $dist(\cdot,Y)^2$ is, up to some multiplicative constant, the restriction of the Hamiltonian function $H_{can}$ from Theorem \ref{thm:canonicalhmae} to the fiber $X_1\simeq X$ (and we shall denote this restriction by $h$).   Thus we actually have a flow of subsets 
\begin{equation}
\tilde{U}_{\lambda} : = \{ x\in X : h(x)\le \lambda\} \text{ for } \lambda\le \Lambda_{can}\label{eq:deftildeU}
\end{equation}
along with canonical tubular neighbourhoods
$$ T_{\lambda} \colon U_{\lambda} \simeq \tilde{U}_{\lambda}$$
from some subset $U_{\lambda}$ of $N_Y$, with the properties of Theorem \ref{thm:tubularmain}.    It turns out that the volume growth of these neighbourhoods is polynomial in $\lambda$, with coefficients that depend only on the topology of $Y$ and the K\"ahler class of $\omega$.

\begin{theorem}\label{thm:volumegrowth}
Assume $X$ is compact, and let $\pi\colon \tilde{X}\to X$ denote the blowup of $X$ along $Y$ with exceptional divisor $E$.  Then the volume of the tubular neighbourhood $T_{\lambda}$ is
$$Vol_{\lambda}:= \int_{\tilde{U}_{\lambda}} \frac{\omega^n}{n!} = \frac{1}{n!}\left(\int_X [\omega]^{n} - \int_{\tilde{X}} (\pi^*[\omega] -\lambda [E])^n\right)$$
where $[E]$ denotes the class of the current of integration along $E$.  In particular $Vol_{\lambda}$ is a polynomial in $\lambda$ whose coefficients depend only on the topology of $Y$ and $X$ and the K\"ahler class of $\omega$.
\end{theorem}

This result bears some similarity to a result of Gray \cite{Gray} who computes the volume growth of the tubular neighbourhood defined using the distance function $dist(\cdot,Y)$ associated to a K\"ahler metric $\omega$.  Gray proves that, if $X$ has constant sectional curvature, then this volume is also a polynomial with topological coefficients (although the coefficients are different to ours).  It may be interesting to ask how these are related, but we will not pursue this question here.

\subsection{Optimal regularity of plurisubharmonic envelopes}

Very much connected to the above discussion are some naturally defined envelopes that occur in pluripotential theory, and our methods provides a new regularity theorem in this context.   Let $Y$ be a compact complex submanifold of a (not necessarily compact) K\"ahler manifold $(X,\omega)$ and let $\lambda>0$ be a parameter.   We consider the envelope
$$\psi_{\lambda}:=\sup\{\psi\in PSH(X,\omega): \psi\leq 0 \text{ and }  \nu_{Y}(\psi)\geq \lambda\}$$
where $PSH(X,\omega)$ denotes the set of $\omega$-plurisubharmonic functions (i.e.\ upper semicontinuous $L^1_{loc}$ functions $\psi$ such that $\omega + dd^c \psi$ is a positive current) and $\nu_Y(\psi)$ denotes the Lelong-number of $\psi$ along $Y$.  Given this data, the \emph{equilibrium set} is defined to be
$$S_{\lambda}:=\psi_{\lambda}^{-1}(0)$$
whose complement $S_{\lambda}^c$ is a neighbourhood of $Y$.   As we will see (c.f.\ Theorem \ref{optimalregthm2} and Remark \ref{rmk:tubularequilibrium}) our tubular neighbourhoods $\tilde{U}_{\lambda}$ from \eqref{eq:deftildeU} are precisely the complement of the equilibrium set.  

Thus the above theorems have implications for the structure of this complement, so we introduce the following definition of regularity that captures what we shall prove:

\begin{definition}
We say that $\psi_{\lambda}$ has \emph{optimal regularity} if
\begin{enumerate}
\item $S_{\lambda}$ is smoothly bounded.
\item On $S_{\lambda}^c\setminus Y$ the envelope $\psi_{\lambda}$ is smooth and also
$$ (\omega+dd^c\psi_{\lambda})^{n-1}\neq 0\quad \text{ on } S_{\lambda}^c\setminus Y.$$
\item There is a family $\{D_{q}\}$ of holomorphic discs in $S_{\lambda}^c,$ parametrized by points $q$ in the exceptional divisor in $Bl_Y X,$ such that the restriction of $\omega+dd^c\psi_{\lambda}$ to each $D_q$ is zero.  Moreover each $D_q$ passes through $Y$, the boundary of $D_q$ lies in $\partial S_{\lambda}$ and the family $\{D_q\}$ foliates the blowup of $S_{\lambda}^c$ along $Y$ (see Figure 2). 
\end{enumerate}
\end{definition}

\begin{figure}[htb]
	\centering
\scalebox{.8}{\input{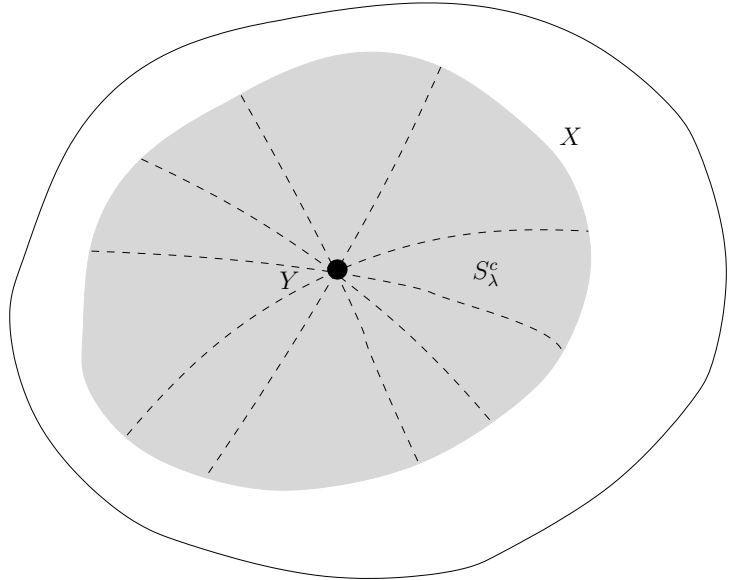}}
	\caption{Optimal Regularity around a submanifold $Y$. Dotted lines represent the discs $D_q$.}
	\label{Optimal Regularity}
\end{figure}

Recall that $\Lambda_{can}$ denotes the canonical radius.

\begin{theorem} \label{optimalregthm}
Assume that $Y$ is compact (or more generally $\Lambda_{can}>0$). Then for $\lambda$ small enough (i.e. $\lambda <\Lambda_{can}$) the envelope $\psi_{\lambda}$ has optimal regularity. Moreover the corresponding holomorphic discs $D_p$ all have area $\lambda$, and the boundaries $\partial S_{\lambda}$ vary smoothly with $\lambda$. 
\end{theorem}  

In fact $\partial S_\lambda=H_{can}^{-1}(\lambda)\cap V_1$ where $H_{can}$ is the Hamiltonian function of the canonical solution $(V_{can},\Omega_{can})$ to the HMAE associated to the data $Y\subset (X,\omega)$, and the foliation of $S_{\lambda}^c$ is the projection of the Monge-Amp\`ere foliation.

The above result is interesting even in the simplest case that $Y$ is a point in an open set in $\mathbb C^n$. This is then a purely local statement and we shall use the optimal regularity to prove the existence of a flow in $\mathbb C^n$ by sets with a certain reproducing property:

\begin{theorem}\label{thm:flowpoint}
Let $\phi$ be a smooth strictly plurisubharmonic function on the open unit ball $B$ in $\mathbb{C}^n$. Let $\psi_{\lambda}$ be defined as
$$\psi_{\lambda}:=\sup\{\psi\in PSH(B): \psi\leq \phi, \nu_0(\psi)\geq \lambda\},$$
and thus the equilibrium set
$$ S_{\lambda} := \{ \psi_{\lambda}= \phi\}.$$
Then for small $\lambda$ (i.e. $\lambda<\Lambda_{can}$) the envelope $\psi_{\lambda}$ has optimal regularity. Also for any bounded holomorphic function $f$ on $B_{\lambda}:=B\setminus S_{\lambda}$ we have
\begin{equation}
\frac{1}{\operatorname{vol}(B_{\lambda})}\int_{B_{\lambda}}f \frac{(dd^c\phi)^n}{n!}=f(0).\label{eq:reproducing}
\end{equation}
\end{theorem}
Here $\Lambda_{can}$ is the canonical radius of the HMAE associated to the data $\{0\}\subset (B,dd^c\phi).$

The reproducing property in \eqref{eq:reproducing} arises in the theory of complex moments and the planar Hele-Shaw flow in fluid mechanics (see Section \ref{sec:previous}) and gives a generalization of this flow to all complex dimensions.

 \subsection{Outline of Proofs} The idea of the proof of Theorem \ref{defconethm} is to begin by picking a point $p\in Y$ and finding a regular solution to the HMAE in a neighbourhood of the proper transform $\mathcal D_p$ of $\{p\}\times D$ inside the deformation to the normal cone $\mathcal N_Y$.    It turns out to be possible, by a simple change of coordinates, to translate this to the Dirichlet problem for the HMAE to one on a product $U\times D$.     The cost in this change is that whereas we originally were seeking a solution with $S^1$-invariant boundary data, the new boundary data will no longer be $S^1$-invariant.   However it will be close to being $S^1$-invariant near $\{p\}\times D$, and we will prove that this is sufficient to ensure that this new Dirichlet problem has a regular solution (the proof relies on a correspondence between solutions to the HMAE and foliations by holomorphic discs with boundary in some given Lagrangian submanifold, and the existence of the latter is unchanged by small pertubations of this submanifold).   Thus for any point $p\in Y$ we get a regular solution to the HMAE near $\mathcal D_p$.

 We then observe that regular solutions to HMAE enjoy some strong uniqueness properties (even if they are defined locally) coming from the existence of the associated foliation.  We use this to show that these local solutions glue together to give a regular solution in a neighbourhood of the proper transform $\mathcal Y$ of $Y\times D$.  The fact that the solutions we produce are cohomologous to $\pi_X^*\omega$ gives the existence of a Hamiltonian function, and we use this to define the notion of canonical radius $\Lambda_{can}$ and to get a canonical regular solution to the HMAE.

To prove that, if $Y$ is compact (or more generally if $\Lambda_{can}>0$), the weak solution agrees with the canonical regular one in a neighbourhood of $\mathcal{Y}$ we use the Legendre transform of the potential $\Phi$ to deduce that on some neighbourhood of $\mathcal{Y}$ this $\Phi$ is bounded from above by a local potential for the weak solution. The fact that $dd^c\Phi$ is harmonic along the leaves of the Monge-Amp\`ere foliation allows us to use the maximum principle to deduce the reverse inequality.

For the result about the envelopes $\psi_{\lambda}$, the key point is that these can be shown to agree with the Legendre transform of the potential for the canonical regular solution. The associated foliation on $\mathcal N_Y$ can then be translated to this optimal regularity result on $X$.

\subsection{Comparison with previous works}\label{sec:previous}  {\bf Tubular Neighbourhoods:}   It is well known that holomorphic tubular neighbourhoods around a complex submanifold $Y\subset X$ need not exist.  In fact if one does exist then the exact sequence $0\to TY\to TX|_Y \to N_Y\to 0$ must split holomorphically.  There are stronger results on non-existence, for instance it is a theorem of Van de Ven \cite{Van} that the only connected submanifolds of projective space that admit holomorphic tubular neighbourhoods are linear subspaces.

The local structure around a holomorphic submanifold $Y$ goes back at least as far as the seminal works of Grauert \cite{Grauert} and Griffiths \cite{Griffiths}, in which some basic notions in algebraic geometry are developed, for example infinitesimal neighbourhoods, notions of positivity of vector bundles and subvarieties, and some vanishing theorems  (see \cite{Movasati} for a survey).   In this, and related works, what is often sought after is a holomorphic transverse foliation by which what is meant a holomorphic family of disjoint subvarieties $S_y$ for $y\in Y$ such that $S_y\cap Y=\{y\}.$ The existence of such a foliation places strong restrictions on the normal bundle. In contrast, the Monge-Amp\`ere foliations used here are typically not holomorphic, and thus no assumptions on the normal bundle are needed. 

Infinitesimal neighbourhoods in algebraic geometry (now called ``formal neighbourhoods") appear from the work of Zariski and Grothendieck.  One of the main uses of the deformation to the normal cone in algebraic geometry is to get around the non-existence of holomorphic tubular neighbourhoods (as used, for example, in intersection theory).  Thus it is not particularly surprising that this same deformation appears here.\vspace{2mm}

\noindent {\bf Geodesics in the space of K\"ahler Metrics:}  The Dirichlet problem for the complex HMAE has a long history, going back at least as far as the fundamental work of Bedford-Taylor \cite{BedfordTaylor}.   The existence of smooth (or regular) solutions is a difficult and much studied problem and can depend in a subtle way on the boundary data (see \cite{Guedjbook} for a survey).   Following work of Semmes \cite{Semmes}, Mabuchi \cite{Mabuchi2}, Donaldson \cite{Donaldson}, it is known that in the case of compact fibres the existence of solutions can be interpreted as geodesic rays in the space of K\"ahler metrics, and the existence of regular solutions has deep implications to the theory of extremal metrics  (for example \cite{ChenTian,RossNystrom, Sturm4,Sturm1,Sturm2, Yanir, Yanir2, SongZelditch} among others).   In the above we consider this problem for the deformation to the normal cone of a submanifold $Y$.   This family is the simplest non-trivial example of a \emph{test-configuration} that lies at the heart of the Yau-Tian-Donaldson conjecture connecting the existence of a constant scalar curvature K\"ahler metric with the algebro-geometric notion of K-stability, and has been studied from a number of points of view (for example \cite{Apostolov, Hwang, Ross1, Ross2, Nystrom}).   So another interpretation of Theorem \ref{thm:mainhmae2} is that the Phong-Sturm geodesic in the space of K\"ahler metrics associated to the test-configuration given by the deformation to the normal cone of a submanifold $Y$ is regular near the orbit of $Y$. \vspace{2mm}

\noindent {\bf Envelopes and Equilibrium Sets:} The kind of envelopes we consider above are basic objects in the study of the Monge-Amp\`ere equation (see for example \cite[Ch. 6]{Klimek}).  They are related to the asymptotic behaviour of the partial Bergman kernel involving holomorphic sections that vanish to a particular order along the submanifold $Y$.  This was first studied in detail in the toric case by  Schiffman-Zelditch \cite{Shiffman} who introduced the name \emph{forbidden region} essentially for the complement of what we call the equilibrium set.  This was then taken up by Berman \cite{Berman2} in the general projective case  (from whom we have taken the terminology \emph{equilibrium set}) who proves, among other things, that the envelopes are $C^{1,1}$, and then again by Berman-Demailly \cite{BermanDemailly} in the setting of big K\"ahler classes.  We refer the reader also to Ross-Singer \cite{RossSinger} and Ross-Witt-Nystr\"om \cite{RossNystrom2} where these envelopes are considered further.

To determine $\partial S_{\lambda}$ and $\psi_{\lambda}$ from $\phi$ as in Theorem \ref{thm:flowpoint} is a \emph{free boundary} type of problem. It is interesting to compare it with its "fixed boundary" analogue. Given a domain $\Omega\subset \mathbb{C}^n$ and a point $z\in \Omega$ the pluricomplex Green function on $\Omega$ with logarithmic pole at $z$ is defined as $$u_z:=\sup\{\psi\in PSH(\Omega): \psi\leq 0, \nu_z(\psi)\geq 1\}.$$ Note that because of the homogeneity of the boundary condition, scaling the size of the pole only has the effect of scaling the original solution. Lempert proved in \cite{Lempert} that if $\Omega$ is strictly convex and smoothly bounded, then the pluricomplex Green function has optimal regularity, in the sense that there exists a family of discs $D_q$, all of which are attached to the boundary of $\Omega$, which foliates $Bl_z\Omega,$ and such that $u_z$ is harmonic along each disc $D_q$ except for having a simple pole at $z.$ This regularity is known not to hold for more general types of domains (see e.g. \cite{BedfordDemailly}), but when $\Omega$ is smoothly bounded and strictly pseudoconvex Blocki \cite{Blocki} has shown that $u_z\in C^{1,1}(\Omega\setminus \{z\})$ (see also \cite{Guan}).\vspace{2mm}

\noindent {\bf The Hele-Shaw flow: } In the particular case that $Y$ is a single point in a Riemann surface $(\Sigma,\omega)$, the tubular neighbourhoods produced here describe a model of the Hele-Shaw flow with empty initial condition and permeability encoded by $\omega$. This has been studied by Hedenmalm-Shimorin \cite{Hedenmalm} for real-analytic hyperbolic Riemann surfaces. Using the Hele-Shaw flow they construct an exponential map, which has the property that the image of any concentric circle intersects orthogonally the image of any ray emanating from zero. One can check that this implies that their exponential map differs from ours, and does not share the properties listed in Theorem \ref{thm:tubularmain}. The Hele-Shaw flow on a general Riemann surface is the topic of previous work of the authors \cite{RossNystromHeleShaw} in which we prove the $n=1$ case of Theorem \ref{thm:flowpoint}.  The link is made through the fact that the Hele-Shaw flow is characterized by its complex moments, or said another way the reproducing property \eqref{eq:reproducing}.  In this way we can consider Theorem \ref{thm:flowpoint} as a generalization of the Hele-Shaw flow to a point in $\mathbb C^n$ (but do not claim that this has any fluids interpretation).\vspace{2mm}

\noindent{\bf Acknowledgements: } We wish to thank Robert Berman and  Bo Berntdsson for their encouragement and input, and to S\'ebastien Boucksom who at an early stage pointed out the connection to tubular neighbourhoods in the symplectic category. We also thank the referees for their close reading and constructure comments.  During this work JR was supported by an EPSRC Career Acceleration Fellowship (EP/J002062/1). DWN has received funding from the People Programme (Marie Curie Actions) of the European Union's Seventh Framework Programme (FP7/2007-2013) under REA grant agreement no 329070, and previously by Chalmers University and the University of Gothenburg.

\section{Preliminaries}
\subsection{Notation and Terminology}
\newcommand{\smooth}{smooth}
\noindent{\bf Deformation to the normal cone:} Let $Y$ be a complex submanifold of a complex manifold $X$ (which we always take to be connected, closed as a subset and not equal to $X$).    The \emph{deformation to the normal cone} $\mathcal N_Y$ of $Y$ is the blowup 
$$\mathcal N_Y \stackrel{\pi}{\to} X\times D$$ 
of $X\times D$ along $Y\times \{0\}$.  The \emph{exceptional divisor} $E=\pi^{-1}(Y\times \{0\})$ is isomorphic to the projective completion $\mathbb P(N_Y\oplus\mathbb C)$ of the normal bundle $N_Y$ of $Y$.   The map $\pi$ is an isomorphism away from $E$ and the \emph{proper transform} of a subvariety $Z\subset X\times D$ not contained in $Y\times \{0\}$  is the closure (in the classical topology) of $\pi^{-1}(Z \setminus Y\times\{0\})$ in $\mathcal N_Y$. We also let $\pi_X$ and $\pi_D$ denote the composition of $\pi$ with the projections to $X$ and $D$ respectively.



\noindent {\bf Plurisubharmonic functions: } The following notions can be found in \cite{Demailly2,Klimek}.  For an open subset $U$ of a complex manifold $X$ we denote by $\PSH(U)$ the space of plurisubharmonic functions $\phi\colon U\to \mathbb R\cup \{-\infty\}$ on $U$.  We say a function $\phi$ on $U$ is \emph{pluriharmonic} if both $\phi$ and $-\phi$ are plurisubharmonic, and $\phi$ is \emph{strictly plurisubharmonic} if its curvature $dd^c\phi$ is a strictly positive current. Any K\"ahler form $\omega$ can be written locally as $dd^c\phi$ for some smooth strictly plurisubharmonic function $\phi$, which is uniquely defined up to addition by pluriharmonic functions.  Given a closed real $(1,1)$-form $\alpha$ on $X$ we let $PSH(\alpha)$ be the set of upper-semicontinuous $L^1_{loc}$ functions $\psi$ such that $\alpha + dd^c \psi$ is a positive current.  So if $\alpha = dd^cu$ on $U$ then $\psi\in\PSH(\alpha)$ if and only if $u + \psi \in PSH(U)$.

\noindent {\bf Foliations: }  A complex $1$-dimensional foliation on a $n$-dimensional complex manifold $X$ consists of a set $\mathcal F$ of (not necessarily closed) $1$-dimensional complex submanifolds of $X$ called \emph{leaves} that cover $X$ such that the following holds:  there is a cover of $X$ by charts $\{U_\alpha\}$ and coordinate projections $f_{\alpha}\colon U_{\alpha}\to \mathbb R^{2n-2}$ that form a $\Gamma_{2n-2}$ structure in the sense of Haefliger \cite{Haefliger} such that for each $\mathcal L\in \mathcal F$ the set $\mathcal L\cap U_{\alpha}$ is the union of connected components of $f_\alpha^{-1}(c)$ \cite[p547]{BedfordKalka}.  By the Frobenius Theorem, any involutive $k$-dimensional distribution of the tangent bundle of $X$ is integrable, and thus this distribution is in fact the tangent space to the leaves of a uniquely defined foliation.  Moreover if this distribution is complex (i.e.\ consists of complex subspaces of the tangent bundle) then the leaves of this foliation are complex submanifolds by the Theorem of Levi-Civita (see the Appendix of \cite{Klimek}).



\subsection{The Dirichlet problem for the Homogeneous Monge-Amp\`ere Equation}

Suppose that $\pi'\colon V'\to \mathbb C$ is a surjective map from a complex $(n+1)$-dimensional manifold $V'$ to $\mathbb C$ such that all the fibers $V'_{\tau}: = \pi'^{-1}(\tau)$ are manifolds for all $\tau \in \mathbb C$.  We let $V: = \pi^{-1}(D)$ and $\pi:= \pi'|_{V} \colon V\to D$.  When dealing with functions (or forms) on $V$ we consider it as a $(2n+2)$-real dimensional smooth manifold with boundary.  In this paper $V$ will always be either a product $V=X\times D$ or a subset of the deformation to the normal cone $\mathcal N_Y$ of a submanifold $Y\subset X$ (in which case $V'$ is $X\times \mathbb C$ or an open subset of the blowup of $X\times \mathbb C$ along $Y\times \{0\}$ with the obvious projection). 

\begin{definition}
We shall refer to a smooth family of K\"ahler forms $\omega_{\tau}$ on $V_{\tau}$ for $\tau\in S^1$ as \emph{boundary data} for $V$. If $V$ is a subset of $\mathcal{N_Y}$ then the K\"ahler form $\omega$ on $X$ induces the boundary condition $\omega_{\tau} = \omega_{|V_{\tau}}$ which we refer to as the \emph{boundary condition induced by} $\omega$.     
\end{definition}

\begin{definition}
We say that a smooth closed real $(1,1)$-form $\Omega$ on $V$ is a \emph{regular solution to the HMAE} with boundary data $\omega_{\tau}$ if
\begin{enumerate} \label{eq:hmae2}
\item $\Omega_{|V_{\tau}}$ is K\"ahler for all $\tau\in D,$
\item $\Omega_{|V_\tau}=\omega_{\tau}$ for all $\tau\in S^1$, and
\item $\Omega^{n+1}=0$ on $V$.
\end{enumerate}
\end{definition}


The distribution defined by the kernel of a regular solution $\Omega$ to the HMAE is integrable (since $\Omega$ is closed), complex (since $\Omega$ is $(1,1)$), and one-dimensional and transverse to the fibers (since $\Omega$ is nondegenerate along the fibers) (see \cite{BedfordKalka} or \cite{Donaldson}). Thus by the Frobenius integrability theorem there is $1$-dimensional complex foliation of $V$, and by construction the restriction of $\Omega$ to each such leaf vanishes. We shall refer to this as the \emph{Monge-Amp\`ere foliation} determined by $\Omega$.  


\subsection{A local existence result for the HMAE}

We now present a local existence result for the HMAE.  Our proof uses the connection between such solutions and foliations by holomorphic discs,   and is based on the Donaldson's proof \cite{Donaldson} of the following theorem:

\begin{theorem}\label{thm:donaldsonopenness}
Let $(X,\omega)$ be compact K\"ahler and set $V=X\times D$. Suppose that for some function $F(z,\tau)$ on $V$ the form $\Omega_F=\pi_X^*\omega+dd^cF$ is a regular solution to  the HMAE  with boundary data $\omega_{\tau}:=\omega+dd^cF(\cdot,\tau)$. Then any smooth function $g = g(z,\tau)$ on $X\times S^1$ that is sufficiently close to $F_{|X\times S^1}$ with respect to the $C^\infty$-topology can be uniquely extended to smooth function $G$ on $V$ such that $\Omega_G=\pi_X^*\omega+dd^cG$ is a regular solution to  the HMAE  with boundary data $\omega + dd^c g(\cdot,\tau)$.
\end{theorem}

We recall some ingredients of the proof.    Semmes \cite{Semmes} constructs a fiber bundle $W_X$ over $X,$ locally modelled on the cotangent bundle of $X,$ such that K\"ahler metrics of $X$ cohomologuous to $\omega$ correspond precisely to certain $LS$-submanifolds of $W_X.$ The boundary data $\omega_{\tau}:=\omega+dd^cF(\cdot,\tau)$ then correspond to a $LS$-submanifold $\Lambda_{\tau}$ of $W_X\times D.$ Donaldson proves in \cite{Donaldson} that regular solutions to the HMAE over $X\times D$ correspond to certain families of holomorphic discs in $W_X\times D$ that cover $D$ and attach to $\Lambda_{\tau}$ along their boundary. The projection of these discs to $X\times D$ will form the Monge-Amp\`ere foliation associated to the regular solution. The proof of Theorem \ref{thm:donaldsonopenness} then relies on the deformation theory for holomorphic discs attached to totally real submanifolds (an $LS$-submanifold is in particular totally real).

By using similar ideas we next prove a local existence theorem for regular solutions to the HMAE in a particular setting (what we do here is somewhat easier in that it only involves rather specific deformations of the $LS$-submanifold).

\begin{theorem}\label{thm:localexistence}
Let $\phi$ be a smooth strictly plurisubharmonic function on the open unit ball $B_1\subseteq \mathbb{C}^n$ and assume that
\begin{equation} \label{locexisteq}
\phi(z)=|z|^2+O(|z|^3).
\end{equation}
Let $\phi_{\tau}(z):=\phi(\tau z',z'')$ for $\tau\in S^1$ where $z'=(z_1,...,z_r)$ and $z''=(z_{r+1},...,z_n)$ for some number $r$. Then there exists a $\delta\in (0,1]$ and a regular solution $\Omega$ to the HMAE on $B_{\delta} \times D$ with boundary data $\omega_{\tau}:=dd^c\phi_{\tau}(z)|_{B_{\delta}}$. The solution $\Omega$ is invariant under the following $S^1$-action $\rho$:
\[ \rho(e^{i\theta}) \cdot (z',z'',\tau):= (e^{i\theta} z',z'',e^{-i\theta}\tau) = (e^{i\theta} z_1,\ldots,e^{i\theta} z_r,z_{r+1},\ldots,z_n,e^{-i\theta}\tau).\] Furthermore there exists a smooth function $\Phi'$ on $B_{\delta}\times D$ such that $dd^c\Phi'=\Omega$ and $\Phi'(z,\tau)=\phi_{\tau}(z)$ for $\tau\in S^1.$
\end{theorem}  

\begin{proof}
Our K\"ahler manifold is in this case $B_1$ and the fiber bundle $W_X$ is identified with $B_1\times \mathbb{C}^n.$ Let $w$ denote the variables of the fiber $\mathbb{C}^n$. The $LS$-submanifold $\Lambda_{\tau}$ corresponding to the boundary data $\phi_{\tau}(z)$ is the graph of $\frac{\partial}{\partial z}\phi_{\tau}(z)$ over $B_1\times S^1.$ We first note that $$\tau \mapsto (0,0,\tau)\in B_1\times \mathbb{C}^n\times D$$ is a holomorphic disc which attaches along its boundary to $\Lambda_{\tau}$ thanks to the assumption (\ref{locexisteq}).

Consider the case when $\phi(z)=|z|^2.$ Then for any $z\in B_1$ we have a holomorphic disc $$\tau\mapsto (z,\bar{z},\tau)$$ which attaches to $\Lambda_{\tau}$ along its boundary.  These form a $2n$-dimensional family of such discs, whose projections to $B_1\times D$ forms a foliation (the standard one). Donaldson proves in \cite[p192]{Donaldson} that in this case the discs are superregular (we refer the reader to \cite[Definition 1]{Donaldson} for the definition of superregular). In particular they are regular and have index $2n,$ so any holomorphic disc nearby $\tau \mapsto (0,0,\tau)$ that attaches to $\Lambda_{\tau}$ along its boundary must be of the form $\tau \mapsto (z,\bar{z},\tau)$ (which can also be seen directly).

The property of a disc to be superregular only depends on the tangent space of $\Lambda_{\tau}$ inside the total tangent bundle restricted to the disc. Since this is unchanged for the disc $\tau \mapsto (0,0,\tau)$ when $\phi(z)=|z|^2$ is replaced with $\phi(z)=|z|^2+O(|z|^3)$ we conclude that the disc is still superregular.  Thus it belongs to a family of discs attaching to $\Lambda_{\tau}$ whose projections foliate an open subset $V\subset B_1\times D$. Since $\phi_{\tau}$ is invariant under the $S^1$-action we can assume that $V$ is as well. We can also assume that each fiber $V_{\tau}$ is open and diffeomorphic to a ball.

We now argue exactly as in \cite[Proposition 1, Lemma 3]{Donaldson} to deduce the existence of a regular solution $\Omega$ to the HMAE on some $B_{\delta}\times D \subseteq V$ with the correct boundary data. The $S^1$-invariance of the boundary data and uniqueness of the discs ensure that $\Omega$ will be $S^1$-invariant.

For the final statement, by the $dd^c$-lemma we can find some function $\Phi''$ on $B_{\delta}\times D$ such that $dd^c\Phi''=\Omega,$ and we can assume $\Phi''$ to be $S^1$-invariant. We then have that $h(z):=\Phi''(z,1)-\phi(z)$ is pluriharmonic on $B_{\delta}$.   It now follows that $$\Phi'(z,\tau):=\Phi''(z,\tau)-h(\tau z)$$ has the desired properties. 
\end{proof}

\section{Local Regular Solutions}

\subsection{Local Existence}\label{sec:local}

Let $Y$ be a complex submanifold of a K\"ahler manifold $(X,\omega)$ and consider the family $\mathcal N_Y\stackrel{\pi}{\to} X\times D$ given by the deformation to the normal cone of $Y$.   We recall that we say that a leaf of a one dimensional foliation on a subset of $\mathcal N_Y$ is \emph{complete} if the projection by $\pi_D$ from this leaf to the base $D$ is a surjection.  

\begin{definition}
For $p\in Y$ we let $\mathcal D_p$ denote the proper transform of $\{p\}\times D$ in $\mathcal N_Y$.
\end{definition}

\begin{definition}\label{def:localregularsolution}
A \emph{local regular solution} to  the HMAE  with boundary data induced by $\omega$ for a subset $A\subset Y$ consists of a pair $(V,\Phi)$ such that
  \begin{enumerate}
  \item $V$ is an open $S^1$-invariant neighbourhood of   the closure (in the classical topology) of $\pi^{-1}(A\times D\setminus Y\times \{0\})$,     that does not meet the singular locus of $\pi_D^{-1}(0)$ and so all fibers $V_{\tau}$ for $\tau\in D$ are smooth manifolds,
  \item $\Phi$ is a smooth real valued $S^1$-invariant function on $V$ which is zero on $V_{\tau},$ $\tau\in S^1,$
  \item $\Omega:=\pi_X^*\omega+dd^c\Phi$ is a regular solution to the HMAE on $V$, 
  \item $V$ is the union of complete leaves of the Monge-Amp\`ere foliation determined by $\Omega$, and if a leaf intersects $\mathcal{Y}$ then it must be equal to $\mathcal{D}_p$ for some $p\in Y$.
  \end{enumerate}
\end{definition}

Note that it follows from the definition that if $(V,\Phi)$ is a local regular solution to  the HMAE  for $A\subseteq Y$ then $\mathcal{D}_p$ is part of the Monge-Amp\`ere foliation for each $p\in A$.

\begin{proposition}\label{prop:hmaelocal}
Let $p\in Y$.  Then there exists a local regular solution $(V,\Phi)$ to  the HMAE  for the subset $\{p\}$.
\end{proposition}

As discussed previously, the idea is to reduce the problem to finding regular solutions to  the HMAE  over a neighbourhood of $\mathcal{D}_p$ which looks like a ball times the disc, i.e.\ the product case, but with a twisted boundary condition.     To set this up, let $z_1,\ldots,z_n$ be coordinates centered at $p$ chosen so that $Y$ is given locally by $z_1=\cdots=z_r=0$ (these coordinates will later be chosen with additional properties).      For simplicity we shall write $z=(z',z'')$ where $z'=(z_1,\ldots,z_r)$ and $z''=(z_{r+1},\ldots,z_n)$. Let $f:B_{\delta}\to X$ denote the corresponding holomorphic embedding of the ball of radius $\delta,$ where $\delta$ is chosen small enough so that $Y$ is given by $z_1=\cdots=z_r=0$ on the chart $f^{-1}:U\to B_{\delta}$.  Let also $\omega_f:=f^*\omega$.  

Consider the map $\Gamma\colon B_{\delta}\times D^{\times} \to U\times D^{\times}$ given by
\[ \Gamma(z,w)  = (f(wz',z''),w).\]


Considering $U\times D^{\times}$ embedded in $\mathcal N_Y$ by $\pi^{-1}$, we see from the definition of the blowup that this map extends to a biholomorphism from $B_{\delta}\times D$ to an open set $V\subset \mathcal N_Y$ that contains $\mathcal{D}_p$.   Henceforth we shall let
\begin{equation}
  \label{eq:definitionV'}
\Gamma \colon B_{\delta}\times D \to V
\end{equation}
denote this extended biholomorphism. Thus $\Gamma^{-1}$ is a holomorphic chart on $\mathcal N_Y$ containing $\mathcal D_p$.

Recall the holomorphic action $\rho$ of $S^1$ on $B_{\delta}\times D:$
\[ \rho(e^{i\theta}) \cdot (z',z'',\tau):= (e^{i\theta} z',z'',e^{-i\theta}\tau) = (e^{i\theta} z_1,\ldots,e^{i\theta} z_r,z_{r+1},\ldots,z_n,e^{-i\theta}\tau).\]
Under the biholomorphism $\Gamma$ the action $\rho$ is the same as the $S^1$ action on $\mathcal{N_Y}$ considered above restricted to $V=\Gamma(B_{\delta}\times D)$. We shall therefore refer to this data as a $\rho$-chart. By abuse of notation we will let the holomorphic action of $S^1$ on $B_1$ given by $$e^{i\theta}\cdot(z',z''):=(e^{i\theta}z',z'')$$ also be denoted by $\rho$.

\begin{proposition}\label{prop:translation}
There exists a local regular solution to  the HMAE  for the subset $\{p\}$ if and only if there exists a $\rho$-chart $f^{-1}:U \to B_{\delta}$ centered at $p$ together with a smooth real $\rho$-invariant function $\Phi'$ on $B_{\delta}\times D$ such that $\Omega':=dd^c\Phi'$ is a regular solution to the HMAE with boundary condition $\omega_{\tau}:=\rho(\tau)^*\omega_f$. 
\end{proposition}

\begin{proof}
If $(V,\Phi)$ is a local regular solution to  the HMAE  for the subset $\{p\}$ then pick a $\rho$-chart $f^{-1}:U \to B_{\delta}$ centered at $p$ such that $\Gamma(B_{\delta}\times D)\subseteq V$. Let $\Phi'(z,\tau):=\phi\circ \rho(\tau)(z)+\Phi\circ \Gamma(z,\tau),$ where $dd^c\phi=\omega_f$. Since the usual $S^1$-action on $\mathcal{N_Y}$ restricts to $\rho$ on the coordinate chart $B_{\delta}\times D,$ $\Phi'$ is $\rho$-invariant, and since $\Gamma$ is a biholomorphism $\Omega'$ is a regular solution to the HMAE on $B_{\delta}\times D$. One easily checks that it has boundary values $\omega_{\tau}:=\rho(\tau)^*\omega_f$.

For the other direction, let us assume that we have a $\rho$-chart and an $\rho$-invariant function $\Phi'$ on $B_{\delta}\times D$ such that $\Omega':=dd^c\Phi'$ is a regular solution to the HMAE with boundary condition $\omega_{\tau}:=\rho(\tau)^*\omega_f$. Let $\phi:=\Phi'_{|B_{\delta}\times \{1\}}$. As the boundary data is invariant, the same is true for the associated Monge-Amp\`ere foliation of $B_{\delta}\times D$.  In particular $\{0\}\times D$ must be a leaf of this foliation, so by continuity there exists a neighbourhood $V'$ of $\{0\}\times D$ consisting of complete leaves. Since the the foliation is $\rho$-invariant we can take the union of all $\rho(e^{i\theta})V'$ which is then a $\rho$-invariant neighbourhood $V'$ of $\{0\}\times D$ consisting of complete leaves, and we now call this neighbourhood $V'$. Because $\Gamma$ is a biholomorphism, $\Gamma^* \omega = \rho(\tau)^*\omega_f$ on $B_{\delta}\times S^1$ and $\rho$ corresponds to the restriction of the usual $S^1$-action on $\mathcal{N_Y}$ to $V=\Gamma(V')$ we get that $(V,\Phi)$ where $\Phi:=\Phi'\circ\Gamma^{-1}-\phi\circ f^{-1}\circ \pi_X$ is a local regular solution to  the HMAE  for the subset $\{p\}$. 
\end{proof}

\begin{lemma}\label{lemma:coordinates}
  There exists a $\rho$-chart centered at $p$ such that $\omega_f=dd^c \phi$ with
$$\phi(z) =|z|^2 + O(|z|^3).$$
\end{lemma}
\begin{proof}


For completeness we include the standard argument.   Let $u=(u',u'')$ with $u'=(u_1,\ldots,u_r)$ and $u''=(u_{r+1},\ldots,u_n)$ be coordinates centered at $p$ defining some $\rho$-chart, and thus $Y$ is given locally by $u'=0$. Also choose a potential $\psi$ for $\omega_f$. The fact that $\psi$ is real valued implies that the Taylor expansion around $0$ takes the form
$$\psi(u) = \alpha + Re(\sum_i \alpha_i u_i + \beta_i u_i^2) + \sum_{ij} \gamma_{ij} u_i\bar{u}_j + o(|u|^3)$$
for some real coefficients $\alpha,\alpha_i,\beta_i$ with 
$$\gamma_{ij} = \frac{\partial^2 \psi}{\partial u_i \partial \overline{u}_j}|_{u=0}.$$
Clearly $h(u):= \alpha + Re(\sum_i \alpha_i u_i + \beta_i u_i^2)$ is pluriharmonic since it is the real part of a holomorphic function. 
 
As $\psi$ is strictly plurisubharmonic, $\Gamma:=(\gamma_{ij})$ is a positive definite hermitian.  Thus by choosing an orthonormal basis with respect to the hermitian form given by $\Gamma$ one gets a linear change of coordinates $P$ so that $P^* \Gamma P=I$. Moreover this can be achieved whilst preserving a given subspace (say by first picking such an orthonormal basis for this subspace).  Thus there exits a matrix with block form
$$ P := \left(
  \begin{array}{ll}
 * & 0 \\
 * & *
  \end{array}
\right)$$
so that $P^*\Gamma P=I$. Letting $z:=Pu$ we get a new $\rho$-chart $f'$ with $Y = \{z'=0\}$ and
$$\psi(P^{-1}z) = h(P^{-1}z) + |z|^2 + O(|z|^3).$$
Now $\omega_{f'}=dd^c\phi$ where $\phi(z):=\psi(P^{-1}z)-h(P^{-1}z)$ which clearly is of the right form.
\end{proof} 
 

\begin{proof}[Proof of Proposition \ref{prop:hmaelocal}]
The existence of a local regular solution for the subset $\{p\}$ follows immediately on combining Proposition \ref{prop:translation}, Lemma \ref{lemma:coordinates} with Theorem \ref{thm:localexistence}.
\end{proof}

\subsection{Patching}\label{sec:patch}

We now show how the local solutions provided in Section \ref{sec:local} patch together. This will follow from a local uniqueness property for these regular solutions.

\begin{proposition} \label{prop:localuniq}
If $(V^{\alpha},\Phi^{\alpha})$ is a local regular solution to the HMAE for a subset $A\subseteq Y$ and $(V^{\beta},\Phi^{\beta})$ is a local regular solution to the HMAE for a subset $B\subseteq Y$ then in fact $\Phi^{\alpha}=\Phi^{\beta}$ in some neighbourhood of the closure of $\pi^{-1}( A\cap B\times D \setminus Y\times \{0\})$.
\end{proposition} 

\begin{proof}
Pick a point $p\in A\cap B$ and a $\rho$-chart $f^{-1}:U\to B_{\delta}$ centered at $p$ such that $\Gamma(B_{\delta}\times D)\subseteq (V^{\alpha}\cap V^{\beta})$.   Let $\phi$ be a function such that $dd^c\phi=\omega_f$ and let $\Phi':=\phi\circ \rho(\tau)(z)+\Phi^{\alpha}\circ \Gamma(z,\tau)$ and $\Phi'':=\phi\circ \rho(\tau)(z)+\Phi^{\beta}\circ \Gamma(z,\tau)$.   Then $\Phi'$ and $\Phi''$ are two $\rho$-invariant regular solutions to the HMAE on $B_{\delta}\times D$ with the same boundary values. Let $\mathcal{L}$ be a complete leaf of the foliation associated to $\Omega',$ clearly we have that $\Phi'$ is harmonic along $\mathcal{L}$.   Since $\Phi'=\Phi''$ on the boundary of $\mathcal{L}$ while $\Phi''$ is subharmonic on $\mathcal{L}$ it follows from the maximum principle that $\Phi'\geq \Phi''$ on $\mathcal{L}$.  We know that the constant disc $\{0\}\times D=\Gamma^{-1}(\mathcal{D}_p)$ is part of the Monge-Amp\`ere foliation of $\Omega':=dd^c\Phi'$ ( as well as that of $\Omega'':=dd^c\Phi''$) and thus a neighbourhood of $\{0\}\times D$ is foliated by such complete leaves, so in this neighbourhood $\Phi'\geq \Phi''$.   But arguing the same way using complete leaves of the foliation associated to $\Omega''$ now gives us that in fact $\Phi'=\Phi''$ in a neighbourhood of $\{0\}\times D$. Pulling back by $\Gamma^{-1}$ yields the identity of $\Phi^{\alpha}$ and $\Phi^{\beta}$ near $\mathcal{D}_p,$ and we are done.
\end{proof} 

\begin{remark}
For spaces $\pi_D: V\to D$ where the fibers are compact manifolds without boundary there exist general uniqueness results for the Homogeneous Monge-Amp\`ere equation even for weak solutions. However the local uniqueness result stated here is quite different in that it does not depend on the behaviour of the solution near the boundary of $V$.
\end{remark}

If $(V,\Phi)$ is a local regular solution to the HMAE and $C\subseteq V_1$ then we let $V_C$ be defined as the union of all leaves of the associated Monge-Amp\`ere foliation which connect to $C\times \{1\}\subseteq \mathcal{N_Y}$.  Note that this is an $S^1$-invariant subset of $\mathcal{N_Y}$.  Given $\epsilon>0$ we will let $C^{\epsilon}$ denote the set of points in $X$ within an $\epsilon$ distance of $C$ (measured using the K\"ahler metric).

\begin{lemma} \label{lem:locpatch}
If $A,B$ are two compact subsets of $Y$ and $(V^{\alpha},\Phi^{\alpha})$ is a local regular solution to the HMAE for $A$ while $(V^{\beta},\Phi^{\beta})$ is a local regular solution to the HMAE for $B,$ then there exists a local regular solution $(V,\Phi)$ to the HMAE for $A\cup B$.   Moreover if $C$ is a compact subset of $X$ such that $C\cap Y=A$ then we can choose $(V,\Phi)$ so that for some $\epsilon>0$ we have that $V^{\alpha}_{C^{\epsilon}}\subseteq V$ and $\Phi=\Phi^{\alpha}$ on $V^{\alpha}_{C^{\epsilon}}$.
\end{lemma}

\begin{proof}
By Proposition \ref{prop:localuniq}, for any point $p\in A\cap B$ we can pick an $S^1$-invariant neighbourhood $U_p$ of $D_p$ in $V^{\alpha}\cap V^{\beta}$ where $\Phi^{\alpha}=\Phi^{\beta},$ consisting only of complete leaves of the associated foliation. Let $U$ denote the union of those neighbourhoods. Now we pick a compact subset $C\subseteq V_1^{\alpha}$ such that $C\cap Y=A$.   By property (4) of a local regular solution (see Definition \ref{def:localregularsolution}) we have that $V^{\alpha}_{C^{2\epsilon}}\cap \mathcal{Y}$ is equal to the closure of $\pi^{-1}(C^{2\epsilon}\cap Y\times D \setminus Y\times \{0\})$.   This implies that when $\epsilon$ is small enough  $$V^{\alpha}_{C^{\epsilon}}\cup U\cup (V_{\beta}\setminus V^{\alpha}_{C^{2\epsilon}})$$ is a neighbourhood of the closure of $\pi^{-1}(A\cup B)\times D\setminus Y\times \{0\})$.  We then define $\Phi$ to be equal to $\Phi^{\alpha}$ on $V^{\alpha}_{C^{\epsilon}},$ while letting it be $\Phi^{\beta}$ on $(V_{\beta}\setminus V^{\alpha}_{C^{2\epsilon}}),$ and then equal to either one on $U$. This now gives local regular solution $(V,\Phi)$ to the HMAE for $A\cup B$.       
\end{proof}

The following gives a proof of Theorem \ref{defconethm}.   As before the data consists of a K\"ahler manifold $(X,\omega)$ (not necessarily compact) together with a complex submanifold $Y\subset X$.

\begin{theorem}\label{thm:hmae:repeat}
There exists a local regular solution $(V,\Phi)$ to the HMAE for $Y$. In particular it means that we have an $S^1$-invariant neighbourhood $V$ of $\mathcal Y$ in $\mathcal N_Y$ and a smooth closed $S^1$-invariant real $(1,1)$-form $\Omega:=\pi_X^*\omega+dd^c\Phi$ on $V$ that gives a regular solution to the homogeneous Monge-Amp\`ere equation with boundary data induced by $\omega$, i.e.
\begin{enumerate}
\item $\Omega_{|V_{\tau}}$ is K\"ahler for all $\tau\in D,$ 
\item $\Omega_{|V_\tau}=\omega_{|V_\tau}$ for all $\tau\in S^1$, and
\item $\Omega^{n+1}=0$ on $V$.
\end{enumerate}
Moreover the germ around $\mathcal Y$ of any such $S^1$-invariant regular solution is unique.
\end{theorem}

\begin{proof}
Using Proposition \ref{prop:hmaelocal} there is a collection $(V^{\alpha_i},\Phi^{\alpha_i}),$ ($i\in I$ with $I=\{1,...,N\}$ or $I=\mathbb{N}$) of local regular solutions to  the HMAE  for compact subsets $A_i\subseteq Y$ such that $\{A_i^{\circ}\}_{i\in I}$ is a locally finite cover of $Y$. Let $C_1:=A_1$. Then by Lemma \ref{lem:locpatch}, for some $\epsilon_1>0$ we can find a locally regular solution $(V^2,\Phi^2)$ to  the HMAE  for $(A_1^{\epsilon_1}\cap Y)\cup A_2$  such that $V^{\alpha_1}_{A_1^{\epsilon_1}}\subseteq V^2$ and so that $\Phi^2=\Phi^{\alpha_1}$ there. For $k\geq 2$ we let $C_k:=C_{k-1}^{\epsilon_k}\cup A_k$ and do the same thing, and we will get a sequence of local regular solutions $(V^k,\Phi^k)$ to  the HMAE  for $\cup_{i=1}^k A_i$. We note that the subsets $V^k_{C_k^{\epsilon_{k+1}}}$ are all $S^1$-invariant and increase with $k,$ and for $l>k,$ $\Phi^l=\Phi^k$ on $V^k_{C_k^{\epsilon_{k+1}}}$. Let $V:=\cup_{k=1}^{\infty}V^k_{C_k^{\epsilon_{k+1}}}$ and define $\Phi$ by letting it be $\Phi^k$ on $V^k_{C_k^{\epsilon_{k+1}}},$ then it is immediate that $(V,\Phi)$ has the properties described in the theorem. This proves the existence part, while the uniqueness of the germ follows directly from Proposition \ref{prop:localuniq}.   
\end{proof}

In fact we can say a bit more than just that the germ of any regular solution is unique, but for this we need to introduce the notion of complete regular solutions.

\begin{definition}
We will call a pair $(V,\Omega)$ a \emph{complete regular solution} to the HMAE (or in short a \emph{complete solution}) if it is a regular solution to the HMAE as in Theorem \ref{defconethm} (so in particular it is cohomologous to $\pi_X^*\omega$) such that $V$ is foliated by complete leaves of the Monge-Amp\`ere foliation, $V_0\subseteq N_Y$ and whenever $u\in V_0$ then $\tau u\in V_0$ for all $\tau \in D$.
\end{definition}

Given a regular solution $(V,\Omega)$ as in Theorem \ref{defconethm}, by shrinking $V$ we can always get a nontrivial complete solution. 

\begin{proposition} \label{prop:completeuniq}
If $(V,\Omega)$ and $(V',\Omega')$ are two complete solutions and $V_0\subseteq V'_0$ then $V\subseteq V'$ and $\Omega=\Omega'$ on $V$. 
\end{proposition}

\begin{proof}
Let $\Phi$ and $\Phi'$ be the potentials of $\Omega$ and $\Omega'$ respectively. Pick a point $u\in V_0$ and let $\mathcal{L}_t$ denote the leaf in the Monge-Amp\`ere foliation of $\Omega$ that passes through $tu$ for $t\in [0,1]$. Let also $\mathcal{L}'_t$ denote the leaf in the Monge-Amp\`ere foliation of $\Omega'$ that passes through $tu$. We know that $\mathcal{L}_0=\mathcal{L}_0'$. Let $T:=\sup\{t: \mathcal{L}_t=\mathcal{L}'_t\}$. We claim that $T=1$. Indeed, if $\mathcal{L}_t=\mathcal{L}'_t$ then there is a neighbourhood $U\subseteq V\cap V'$ of $\mathcal{L}_t$ consisting of complete leaves of the Monge-Amp\`ere foliation of $\Omega$. On any such leaf $\Phi'-\Phi$ is subharmonic, and since this function is zero on the boundary of the leaf we get that $\Phi\geq \Phi'$. Similarly, using a neighbourhood of $\mathcal{L}_t=\mathcal{L}'_t$ consisting of leaves of the foliation associated to $\Omega'$ we get that in fact $\Phi=\Phi'$ in a neighbourhood of $\mathcal{L}_t,$ and hence the foliations agree there. This implies that the set of points $t\in [0,1]$ such that $\mathcal{L}_t=\mathcal{L}'_t$ is open. On the other hand it is closed since the leaves vary continuously with $t$. This thus shows that $\mathcal{L}_1=\mathcal{L}'_1$ and also that $\Phi=\Phi'$ on that leaf. Since $V$ is foliated by such leaves it gives us the proposition.
\end{proof}

\begin{remark}
Note that in the proof we only used the facts that $V$ and $V'$ were foliated by complete leaves, $V_0\subseteq N_Y,$ $V_0$ connected and $V_0\subseteq V'_0$.  
\end{remark}

Let $\zeta$ denote the vector field generated by the $S^1$-action on $V$. From the fact that $\Omega$ is cohomologous to $\pi_X^*\omega$ follows that the $S^1$-action is Hamiltonian, in the sense described in the next theorem. 

\begin{theorem}\label{thm:ham2}
Let $(V,\Omega)$ be a regular solution to the HMAE as in Theorem \ref{defconethm}. Then the function $H:=L_{J\zeta}\Phi$ is a Hamiltonian for the $S^1$-action, in that it satisfies 
\begin{equation} \label{eq:hamiltonian}
dH = \iota_{\zeta} \Omega.
\end{equation} 
$H$ is constant along the leaves of the Monge-Amp\`ere foliation associated to $\Omega$. Furthermore if $(V,\Omega)$ is complete then $H\ge 0$ with equality precisely on $\mathcal Y$.
\end{theorem}

\begin{proof}
From the definition of the $dd^c$-operator we get that $$dL_{J\zeta}\Phi=\iota_{\zeta}dd^c\Phi=\iota_{\zeta}(\Omega-\pi_X^*\omega),$$ but on the other hand clearly $\iota_{\zeta}\pi_X^*\omega=0,$ and thus $H:=L_{J\zeta}\Phi$ is a Hamiltonian for the $S^1$-action.

The Lie derivative of $H$ with respect to the vector field $\xi$ obtained by the flow along a leaf of the associated foliation is
$$ L_{\xi} H =\iota_{\xi} d H = 0$$
where we have used \eqref{eq:hamiltonian} and that by definition $\xi$ lies in the kernel of $\Omega$.  Thus $H$ is constant along leaves.  As $\Phi$ is zero on $\mathcal Y$ and $\zeta$ is tangential to $\mathcal Y$ we get that $H=0$ along $\mathcal Y$.

It remains to prove that $H$ is strictly positive away from $\mathcal Y$ when the solution is complete. Since $H$ is constant along the leaves of the foliation it suffices to consider $H$ on the discs $\{\tau u:u\in V_0\setminus \iota(Y),\tau \in D\}\subseteq V_0$. We have that $$H(u)=\frac{d}{dt}_{t=0}\Phi(e^{t/2} u)$$ and since $\pi_X^*\omega$ restricts to zero on the disc the function $t\mapsto \Phi(e^t u)$ is strictly convex and $0$ at $t=-\infty$, which implies that $H(u)>0$ as long as $u\notin \iota(Y)$. 
\end{proof}

\subsection{Canonical solutions}

We will now show how to define a canonical complete solution $(V_{can},\Omega_{can})$, at least when $Y$ is compact.

\begin{definition}
Given a complete solution $(V,\Omega)$ with Hamiltonian $H$ we define the radius $rad(V,\Omega)$ of $(V,\Omega)$ to be the supremum of all $\lambda\geq 0$ such that $\partial H^{-1}([0,\lambda))\subseteq V$. We then define the \emph{canonical radius} $\Lambda_{can}$ to be the supremum of $rad(V,\Omega)$ over all complete solutions $(V,\Omega)$. 
\end{definition}

Note that the radius of a complete solution could be zero, and hencethe canonical radius $\Lambda_{can}$ could also be zero. But at least when $Y$ is compact we clearly have that $rad(V,\Omega)>0,$ and thus $\Lambda_{can}>0$. 

\begin{lemma} \label{lem:completerestr}
If $(V,\Omega)$ is a complete solution and $\lambda>0$ then $(H^{-1}([0,\lambda)),\Omega)$ is a new complete solution.
\end{lemma}

\begin{proof}
Since $H$ is constant along the leaves we see that $H^{-1}([0,\lambda)$ is still foliated by complete leaves, and from the proof of Theorem \ref{thm:ham2} we see that $H(\tau u)\leq H(u)$ when $\tau\in D,$ showing that the discs $\{\tau u:\tau\in D\}$ lie in $H^{-1}([0,\lambda))$ when $u$ does.
\end{proof}

\begin{proposition} \label{prop:completeuniq2}
If $(V,\Omega)$ and $(V',\Omega')$ are two complete solutions with $$rad(V,\Omega)\leq rad(V',\Omega')$$ then letting $\lambda:=rad(V,\Omega)$ we have that $$H^{-1}([0,\lambda))=(H')^{-1}([0,\lambda))$$ and $\Omega=\Omega'$ there.
\end{proposition}

\begin{proof}
Clearly we can assume that $\lambda>0$ since otherwise the statement is vacuous. By Lemma \ref{lem:completerestr} the solutions $(H^{-1}([0,\lambda)),\Omega)$ and $((H')^{-1}([0,\lambda)),\Omega')$ are still complete, so for ease of notation we may as well assume that $V=H^{-1}([0,\lambda))$ and $V'=(H')^{-1}([0,\lambda))$ (and hence $rad(V',\Omega')=\lambda$). Pick a point $u\in V_0,$ then we know that $H(u)<\lambda$. We now use an argument similar similar to the one in the proof of Proposition \ref{prop:completeuniq}. Thus we let $\mathcal{L}_t$ denote the leaf in the Monge-Amp\`ere foliation of $\Omega$ that passes through $tu$ for $t\in [0,1]$. We also let $\mathcal{L}'_t$ denote the leaf in the Monge-Amp\`ere foliation of $\Omega'$ that passes through $tu,$ if there is any such leaf. We let $T:=\sup\{t: \mathcal{L}_t=\mathcal{L}'_t\}$ and again we claim that $T=1$. That the set of points $t\in [0,1]$ such that $\mathcal{L}_t=\mathcal{L}'_t$ is open follows exactly as before. On the other hand, $$\lim_{t\to T}H'(tu)=\lim_{t\to T}H(tu)<\lambda$$ which since $rad(V',\Omega')=\lambda$ implies that $Tu\in V'$. Then the closedness follows as before since the leaves vary continuously with $t,$ which exactly as in the proof of Proposition \ref{prop:completeuniq} gives us the desired equalities.
 
\end{proof}

Armed with this result it is immediate how to construct a canonical complete solution $(V_{can},\Omega_{can})$ when the canonical radius $\Lambda_{can}$ is positive. Namely, let $(V_i,\Omega_i)$ be any sequence of complete solutions such that $\lambda_i:=rad(V_i,\Omega_i)$ is increasing to $\Lambda_{can}$. Then we define $$V_{can}:=\cup_i H_i^{-1}([0,\lambda_i))$$ and we let $\Omega_{can}$ be defined to be equal to $\Omega_i$ on $H_i^{-1}([0,\lambda_i))$. That this is a well-defined complete solution now follows immediately from Proposition \ref{prop:completeuniq2}. We have thus proved Theorem \ref{thm:canonicalhmae} from the Introduction:   

\begin{theorem} \label{thm:canonicalhmae2}
Let $Y$ be compact (or more generally assume that $\Lambda_{can}>0$). Then there is a unique complete solution $(V_{can},\Omega_{can})$ such that $$rad(V_{can},\Omega_{can})=\Lambda_{can}$$ and $$H_{can}< \Lambda_{can}.$$ This canonical solution is maximal in the following sense. If $(V,\Omega)$ is any other complete solution then for any $\lambda<rad(V,\Omega)$ we have that $$H^{-1}([0,\lambda))=H^{-1}_{can}([0,\lambda))$$ and there $\Omega=\Omega_{can}$.
\end{theorem}

\section{Tubular Neighbourhoods} \label{sec:tubular}

Let $Y$ be a submanifold in a K\"ahler manifold $(X,\omega)$. Recall that $\pi\colon N_Y\to Y$ denotes the normal bundle and $\iota\colon Y\to N_Y$ is the inclusion of $Y$ as the zero section in $N_Y$.

\begin{definition}
A \emph{tubular neighbourhood} of $Y$ in $X$ is a smooth map $T:U \to X$ from an open neighbourhood $U$ of $\iota(Y)$ in $N_Y$ which is a diffeomorphism onto a neighbourhood $T(U)$ of $Y$ with $T\circ\iota = \id_Y$.
\end{definition}

Suppose that $\Omega$ is as provided by Theorem \ref{thm:hmae:repeat}, so is a regular solution to  the HMAE  on a neighbourhood $V\subset \mathcal N_Y$ of the proper transform $\mathcal Y$ of $Y\times D$ with boundary data $\omega$.  Recall that $U:=V_0$ denotes the central fibre of $V$, which is a neighbourhood of $\iota(Y)\subset N_Y$.

\begin{definition}
We let $\omega_{N_Y}$ be the restriction of the regular solution $\Omega$ of  the HMAE  to $U$.
\end{definition}

So $\omega_{N_Y}$ is an $S^1$-invariant K\"ahler form on $U$, and by the uniqueness property of regular solutions, the germ of this K\"ahler form around $\iota(Y)\subset N_Y$ is independent of choice of $\Omega$.

We recall that $V$ was assumed to be a union of complete leaves of the Monge-Amp\`ere foliation determined by $\Omega$.  Thus flowing along these leaves gives an injective smooth map
$$ \hat{T}\colon U\times D \to V$$
such that $\pi_D \hat{T}(u,\tau)=\tau$ and for each $u\in U$ the map $\hat{T}_u(\tau)=\hat{T}(u,\tau)$ is holomorphic. The family of maps $\hat{T}(\cdot,\tau)$ has an additional remarkable property. 

\begin{proposition} \label{prop:pullback}
For any $\tau\in D$ we have that
\begin{equation} \label{eq:pullback}
\hat{T}(\cdot,\tau)^*(\Omega_{|V_{\tau}})=\omega_{N_Y}.
\end{equation}
\end{proposition}

\begin{proof}
This is classical (see e.g. \cite{BedfordKalka} or \cite{Donaldson}) but for the convenience of the reader we give the simple argument here.

When differentiating the left hand side of (\ref{eq:pullback}) by a vector field $v$ on the base one gets $$\hat{T}(\cdot,\tau)^*(L_{\tilde{v}}\Omega_{|V_{\tau}})$$ where $\tilde{v}$ is the unique lift of $v$ to $V$ parallel to the foliation, and at the same time by Cartan's formula $$L_{\tilde{v}}\Omega=\iota_{\tilde{v}}d\Omega+d\iota_{\tilde{v}}\omega=0$$ since $\Omega$ is closed and $\tilde{v}$ lies in the kernel. Thus the left hand side of (\ref{eq:pullback}) is independent of $\tau,$ so letting $\tau$ be zero completes the proof.
\end{proof}

\begin{definition}
  We define $T\colon U\to X$ by 
  \begin{equation}
T(u) = \hat{T}(u,1)\label{eq:defineT}
\end{equation}

\end{definition}

\begin{lemma}\label{lem:jistubular}
   $T:U\to X$ is a tubular neighbourhood of $Y$ and $T^* \omega = \omega_{N_Y}$.  In particular $T^* \omega$ is K\"ahler and $S^1$-invariant.
\end{lemma}
\begin{proof}
By construction $T:U\to X$ is a tubular neighbourhood.  The second statement follows from Proposition \ref{prop:pullback}.
\end{proof}

The next result completes the Proof of Theorem \ref{thm:tubularmain} in the Introduction.

\begin{proposition}
  Let $u\in U$ and $f_u\colon S^1\to X$ be given by $f_u(e^{i\theta}) = T(e^{i\theta}u)$ for $e^{i\theta}\in S^1$.  Then there exists a holomorphic $F_u\colon D\to X$ extending $f_u$ such that
  \begin{equation*}
 F_u(0) =\pi(u) \text{ and } \left[DF_u|_0(\frac{\partial}{\partial x})\right] = u.\label{eq:bound}
 \end{equation*}
\end{proposition}
\begin{proof}
  For $u\in U$ there is a unique holomorphic leaf $\mathcal{L}_u$ of the Monge-Amp\`ere foliation that passes though $u$ and by definition of $T$ it contains the point $(T(u),1)$. Now let $F_u$ be the lift from $D$ to $\mathcal{L}$ composed with the projection to $X$. It is then immediate that $F_u(1)=T(u),$ $F_u(0)=\pi(u)$ and $\left[DF_u|_0(\frac{\partial}{\partial x})\right] = u$. The $S^1$-invariance of the foliation implies that $$\mathcal{L}_u=e^{-i\theta}\cdot \mathcal{L}_{e^{i\theta}u},$$ and thus $(T(e^{i\theta}u),e^{i\theta})\in \mathcal{L}_u$. Thus we get that $F_u$ extends $f_u$.  
\end{proof}

When $Y$ is compact (or more generally $\Lambda_{can}>0$) we clearly get a canonical tubular neighbourhood $T_{can}:U_{can}\to X$ with the desired properties by using the canonical complete solution $(V_{can},\Omega_{can})$. 

We also have the following characterization of when the germ of $T$ is holomorphic at a given point $\iota(p)\in \iota(Y)$.

\begin{proposition} \label{prop:hologerm}
The germ of $T$ is holomorphic around a point $\iota(p)\in \iota(Y)$ if and only if there exists a neighbourhood $U_p$ of $p$ together with a holomorphic $S^1$-action on $U_p$ such that $\omega_{|U_p}$ is invariant, $Y\cap U$ is fixed pointwise and the induced action on $N_{Y\cap U}$ is equal to the usual rotation of its fibers.
\end{proposition}

\begin{proof}
If the germ of $T$ is holomorphic at $\iota(p)$ then we can take $U_p$ to be the image under $T$ of an $S^1$-invariant neighbourhood of $\iota(p)$, and consider the induced action of $S^1$ on the image. This will then have the described properties since $T^*\omega$ is $S^1$-invariant.

Let now $U_p$ be a neighbourhood of $p$ with the properties described above. Choose holomorphic coordinates $z_i$ centered at $p$ such that locally around $p$ we have that $Y$ is given by the equations $z_i=0,$ $1\leq i\leq r$. Let $\sigma$ denote the $S^1$-action on $U_p$. For $1\leq i\leq r$ we let $\tilde{z}_i$ be the holomorphic function defined by $$\tilde{z}_i(x):=\frac{1}{2\pi i}\int z_i(\sigma(\tau)\cdot x)\tau^{-2}d\tau,$$ while for $r<i\leq n$ we let $$\tilde{z}_i(x):=\frac{1}{2\pi i}\int z_i(\sigma(\tau)\cdot x)\tau^{-1}d\tau.$$ Then these coordinates $\tilde{z}_i$ define a $\rho$-chart $f^{-1}: U\to B_{\delta}$ and it is easy to see that on this chart $\rho=\sigma$. Thus by assumption $\omega_f$ is $\rho$-invariant, giving rise to the trivial local regular solution $\pi_{B_{\delta}}^*\omega_f$. The associated foliation is the trivial one, and by uniqueness (Proposition \ref{prop:localuniq}) the Monge-Amp\`ere foliation of any regular solution agrees with this trivial one near $\mathcal{D}_p$. In particular we see that $T$ is holomorphic near $\iota(p)$.
\end{proof}

\section{Local Regularity of Weak Solutions}\label{sec:localregularity}

Recall that one cannot find global regular solutions of the HMAE on $\mathcal{N}_Y$ but instead of looking for local regular solutions one can consider global weak solutions.

If $\Psi$ is a $\pi_X^*\omega$-psh function on $\mathcal{N_Y}$ and $\nu_{E}(\Psi)\geq c$ then $\pi_X^*\omega+dd^c\Psi-c[E]$ is a closed positive current cohomologous to $\pi_X^*\omega-c[E]$. Pick a $c>0$ and let $\Phi_w$ be defined as the supremum of all $\pi_X^*\omega$-psh functions on $\mathcal{N}_Y$ such that $\Psi\leq 0$ on $X\times S^1$ and with $\nu_{E}(\Psi)\geq c$. 

We call the closed positive $(1,1)$-current on $\mathcal{N}_Y:$  $$\Omega_w:=\pi_X^*\omega+dd^c\Phi_w-c[E]$$ the weak solution to the HMAE on $\mathcal{N}_Y$. A motivation for this comes from the following theorem. 

\begin{theorem}
The closed positive $(1,1)$-current $$\Omega_w:=\pi_X^*\omega+dd^c\Phi_w-c[E]$$ solves the weak HMAE on $\mathcal{Y}$ with boundary condition induced by $\omega$, i.e.
\begin{align}\tag{weak-HMAE}\label{eq:wHMAE}
\Omega_{|\pi^{-1}(\tau)} &= \omega \text{ for } \tau \in S^1, \\
\Omega_w^{n+1} &= 0 \text{ on } \pi_D^{-1}(D^\times)\nonumber
\end{align} 
\end{theorem}
\begin{proof}
We first note that $\Phi_w$ is subharmonic on each disc $\mathcal{D}_p,$ $p\in X,$ and bounded from above by zero on their boundaries, which implies by the maximum principle that $\Phi_w\leq 0.$ On the other hand it is clear that $c\ln|\tau|^2$ is a candidate for the supremum, which then gives us that $$c\ln|\tau|^2\leq \Phi_w\leq 0.$$ Thus $\Phi_w(x,\tau)\to 0$ uniformly as $\tau\to 1$, which gives us the boundary value statement of \eqref{eq:wHMAE}.
 
We stress that since $\Phi_w$ is not necessarily smooth, the term $\Omega_w^{n+1}$ is to be taken in the sense of Bedford-Taylor. From the inequality $c\ln|\tau|^2\leq \Phi_w$ we get that $\Phi_w$ is locally bounded away from the central fiber. Now a standard local argument due to Bedford-Taylor \cite{BedfordTaylor} (see also Demailly \cite[12.5]{Demailly}) shows that $\Omega_w^{n+1}=0$ away from the central fiber as claimed.
\end{proof}

\begin{theorem}\label{thm:weakagree}
Assume that $Y$ is compact (or more generally $\Lambda_{can}>0$) and let $(V_{can},\Omega_{can})$ be the canonical complete solution as provided by Theorem \ref{thm:canonicalhmae}. Then 
$$\Omega_w = \Omega_{can}$$ on $H_{can}^{-1}([0,c))$. In particular $\Omega_w$ is regular in a neighbourhood of the proper transform of $Y\times D$.
\end{theorem} 

To prove Theorem \ref{thm:weakagree} we will consider the Legendre transform of $\Phi:=\Phi_{can}$, i.e. the unique $S^1$-invariant function on $V:=V_{can}$ being zero on $V_1$ such that $\Omega_{can}=\pi_X^*\omega+dd^c\Phi$.

So for $0\le \lambda<  \Lambda_{can}$ we consider the Legendre transform
\begin{equation} \label{eq:legtrans}
\alpha_{\lambda}(x):= \inf\{ \Phi(x,\tau)+\lambda \ln |\tau|^2: (x,\tau)\in V\cap (X\times D^{\times})\}.
\end{equation}
Note that the function $g(t):=\Phi(x,e^{-t/2})-\lambda t$ is convex and \begin{equation} \label{eq:Halpha}
g'(t)=H(x,e^{-t/2})-\lambda,
\end{equation}
where $H:=H_{can}$ is the Hamiltonian. Since $H \to \Lambda_{can}$ as $(x,\tau)$ approaches $\partial V$ it follows that the infimum in (\ref{eq:legtrans}) is attained in $V$. 

\begin{proposition} \label{prop:alphalambda}
  For $\lambda<\Lambda_{can}$ the function $\alpha_{\lambda}$ is $\omega$-psh, $\{\alpha_{\lambda}<0\}=H^{-1}([0,\lambda)),$ and $\nu_Y(\alpha_{\lambda})\geq \lambda$. 
\end{proposition}

\begin{proof}
That $\alpha_{\lambda}$ is $\omega$-psh follows from the Kiselman minimum principle \cite{Kiselman}, \cite[ChI 7B]{Demailly2}. Also from (\ref{eq:Halpha}) we see that $\alpha_{\lambda}<0$ precisely on $H^{-1}([0,\lambda))$. For the final statement we will work locally around one of the discs $\mathcal{D}_p$. So let $f^{-1}:U\to B_{\delta}$ be an chart constructed as in Lemma \ref{lemma:coordinates}. Then we can write $\omega_f=dd^c\phi$ where $\phi(z)=|z|^2+O(|z|^3)$ and $$\Phi':=\Phi\circ \Gamma+\phi\circ \rho.$$ We know that $\phi\leq C|z|^2$ on $B_{\delta}$ for some number $C>1$ and it follows that $\Phi'\leq C|z|^2$ on $B_{\delta}\times D$. Since as we recall $$\Gamma(e^{t/2}z',z'',e^{-t/2})=(f(z',z''),e^{-t/2})$$ it follows that \begin{eqnarray*}
\alpha_{\lambda}(f(z))+\phi(z)\leq \inf\{\Phi'(e^{t/2}z',z'',e^{-t/2})-\lambda t: t\in [0,T]\}\leq \\ \leq C(e^T|z'|^2+|z''|^2)-\lambda T=C\delta^2-\lambda\ln(\delta^2-|z''|^2)+\lambda \ln |z'|^2.
\end{eqnarray*} 
where $$T:=\ln\left(\frac{\delta^2-|z''|^2}{|z'|^2}\right),$$ i.e. $(e^{T/2}z',z'')\in \partial B_{\delta}$. This shows that $\nu_Y(\alpha_{\lambda})\geq \lambda$.
\end{proof}

\begin{remark}\label{rmk:alphagrowth}
One can similarly use that $\phi\geq c|z|^2$ for some number $c>0$ to show that in fact $\alpha_{\lambda} = \lambda \ln |z'|^2 + O(1)$.
\end{remark}

\begin{proof}[Proof of Theorem \ref{thm:weakagree}]

In light of Proposition 

The statement of Proposition \ref{prop:alphalambda} says that $\alpha_{\lambda}$ is $\omega$-psh inside $V_1$, and equal to zero near the boundary (since $\lambda$ is chosen strictly less that $\Lambda_{can}$).    Thus we may extend $\alpha_{\lambda}$ to be zero on $X\setminus V_1$ to obtain an $\omega$-psh function on all of $X$ that has  Lelong number along $Y$ greater than or equal to $\lambda$.

Now assume that $\lambda< c$. Then the function $\alpha_{\lambda}(x)+(c-\lambda)\ln|\tau|^2$ is $\pi_X^*\omega$-psh, bounded from above by zero on $X\times S^1$ and has Lelong number at least $c$ along $E$. From the definition of $\Phi_w$ as the supremum of such functions we get that $$\alpha_{\lambda}(x)+(c-\lambda)\ln|\tau|^2\leq \Phi_w.$$ On the other hand, by the involution property of the Legendre transform $$\Phi(x,\tau)+c\ln|\tau|^2=\sup\{\alpha_{\lambda}(x)+(c-\lambda)\ln |\tau|^2: 0\leq \lambda <\min(c,\Lambda)\}$$ on $H^{-1}([0,c)),$ and thus we get that $$\Phi(x,\tau)+c\ln|\tau|^2\leq \Phi_w$$ on $H^{-1}([0,c))$. 

Pick a leaf in the foliation of $(H^{-1}([0,c)),\Omega_{can})$ and let $\psi$ be a function on it such that $dd^c\psi=\pi_X^*\omega_{|\mathcal{L}}$. Since $\Phi_w$ has Lelong number $c$ along $E$ it follows that $(\Phi_w-c\ln|\tau|^2)_{|\mathcal{L}}+\psi$ defines a subharmonic function on $\mathcal{L},$ which is equal to $\psi$ on the boundary. On the other hand we know that $\Phi_{|\mathcal{L}}+\psi$ is harmonic and equal to $\psi$ on the boundary, which implies that in fact $$\Phi_w=\Phi+c\ln|\tau|^2$$ on $H^{-1}([0,c))$. Locally on a $\rho$-chart we would get that $$\Phi_w\circ \Gamma=\Phi\circ\Gamma+c\ln|\tau|^2$$ on $H^{-1}([0,c))\cap B_{\delta}\times D,$ and since in this picture $E$ is the zero fiber this then implies that $$\Omega_w=\Omega_{can}$$ on $H^{-1}([0,c))$.

\end{proof}

\section{Optimal regularity of envelopes}

Recall that the envelope $\psi_{\lambda}$ we wish to consider is $$\psi_{\lambda}:=\sup\{\psi\in PSH(X,\omega): \psi\leq 0,  \nu_{Y}(\psi)\geq \lambda\},$$ and the equilibrium set $S_{\lambda}$ is  $$S_{\lambda}:=\{\psi_{\lambda}=0\}.$$ We also recall the definition of optimal regularity for such envelopes. 

\begin{definition}
We say that $\psi_{\lambda}$ has optimal regularity if $S_{\lambda}$ is smoothly bounded, the function  $\psi_{\lambda}$ is smooth on $S_{\lambda}^c\setminus Y,$ $$ (\omega+dd^c\psi_{\lambda})^{n-1}\neq 0 \text{ on } S_{\lambda}^c\setminus Y,$$
and the blowup of $S_{\lambda}^c$ along $Y$ is foliated by holomorphic discs attached to $\partial S_{\lambda}$ passing through $Y$ such that the restriction of $\omega+dd^c\psi_{\lambda}$ to each such disc vanishes.
\end{definition}

\begin{theorem} \label{optimalregthm2}
Suppose $Y$ is compact (or more generally $\Lambda_{can}>0$). Then for $\lambda$ small enough (i.e. $\lambda<\Lambda_{can}$) the envelopes $\psi_{\lambda}$ have optimal regularity, the boundaries $\partial S_{\lambda}$ vary smoothly with $\lambda$, and the corresponding holomorphic discs attaching to $\partial S_{\lambda}$ and passing through $Y$ all have area $\lambda$.
\end{theorem} 

\begin{proof}
Let $(V_{can},\Omega_{can})$ be the canonical solution to the HMAE  with boundary data induced by $\omega,$ as provided by Theorem \ref{thm:canonicalhmae}, and choose a number $c$ such that $\lambda<c\leq\Lambda_{can}$ (if $\Lambda_{can}$ is finite we can just as well set $c=\Lambda_{can}$). Let $\Phi_w$ be corresponding weak solution as defined in Section \ref{sec:localregularity}. We saw in the proof of Theorem \ref{thm:weakagree} that $$\Phi_w=\Phi+c\ln|\tau|^2$$ in $H^{-1}([0,c))$ (here as we recall $\Phi$ is the unique $S^1$-invariant function on $\mathcal{N}_Y$ equal to zero on $X\times S^1$ such that $\Omega_{can}=\pi_X^*\omega+dd^c\Phi$ and $H$ is the Hamiltonian of $(V_{can},\Omega_{can})$).

We get by definition that $$\psi_{\lambda}+(c-\lambda)\ln|\tau|^2\leq \Phi_w,$$ which implies that 
\begin{equation} \label{eq:psialpha}
\psi_{\lambda}(x)\leq \inf_{\tau\in D^{\times}}\{(\Phi_w(x,\tau)-c\ln|\tau|^2)+\lambda\ln|\tau|^2\}.
\end{equation} 
Since $\Phi_w=\Phi+c\ln|\tau|^2$ in $H^{-1}([0,c))$ we see that the right hand side of (\ref{eq:psialpha}) is equal to $\alpha_{\lambda}$ on $H^{-1}([0,c))$ and since it is less than or equal to zero on $X$ it must be equal to $\alpha_{\lambda}$ on the whole of $X$. Thus $\psi_{\lambda}\leq \alpha_{\lambda},$ but since $\alpha_{\lambda}$ is a candidate for the supremum we get that in fact \begin{equation}
\psi_{\lambda}=\alpha_{\lambda}.\label{eq:psilambdaequalsalphalambda}\end{equation}

Thus from Proposition \ref{prop:alphalambda} we get that the boundary $\partial S_{\lambda}$ is equal to $H^{-1}(\lambda)\cap V_1$. Since $H$ is constant along the leaves of the foliation it follows that $\partial S_{\lambda}$ is the image of the $\lambda$ level set of $H$ on $V_0$ under the tubular neighbourhood $T$. Since $\Omega$ restricted to $V_0$ is K\"ahler and $\zeta$ (i.e. the vector field generating the $S^1$-action) is nonvanishing on $N_Y\setminus \iota(Y)$ we get that $dH=\iota_{\zeta}\Omega$ is nonzero away from $\iota(Y)$. This shows that  $\partial S_{\lambda}$ is smooth and varies smoothly with $\lambda$.

Let now $x$ be a point in $S_{\lambda}^c\setminus Y$. Since $H^{-1}(\lambda)$ is foliated by leaves there is one such leaf $\mathcal{L}$ which intersects the open line segment $\{(x,\tau):\tau\in (0,1)\}$. At the intersection point $(x,\tau)$ we then have that $$\alpha_{\lambda}(x)=\Phi(x,\tau)+\lambda\ln|\tau|^2.$$ On $\mathcal{L},$ $dd^c\Phi=-\pi^*\omega$ and thus the function $$\alpha_{\lambda}(x)-\Phi(x,\tau)-\lambda\ln|\tau|^2$$ restricted to $\mathcal{L}$ is thus subharmonic, bounded from above by zero and equal to zero at an interior point $(x,\tau),$ and thus must be equal to zero. Therefore
\begin{equation} \label{envelopeequal}
\alpha_{\lambda}(x)=\Phi(x,\tau)+\lambda\ln|\tau|^2
\end{equation}
on $H^{-1}(\lambda)$ and
\begin{equation} \label{harmonicity}
dd^c\alpha_{\lambda}(x,\tau)=-\pi_X^*\omega
\end{equation}
on $\mathcal{L}$. This now shows that $\alpha_{\lambda}=\psi_{\lambda}$ is smooth on $S_{\lambda}^c\setminus Y$. If we then project the leaf $\mathcal{L}$ to $X$ we get a holomorphic disc $\mathcal{L}_x$ passing through $Y$ which attaches to $\partial S_{\lambda}$ along its boundary. Also because of (\ref{harmonicity}) we have that $\omega+dd^c\psi_{\lambda}$ vanishes on $\mathcal{L}_x$. It is clear that these discs foliate the blowup of $S_{\lambda}^c$ along $Y$. To calculate the area of one of these discs $\mathcal{L}_x$ we first observe that $$\int_{\mathcal{L}_x}\omega=\int_{T^{-1}(\mathcal{L}_x)}\omega_{N_Y}.$$ Now $T^{-1}(\mathcal{L}_x)$ is not necessarily a holomorphic disc in $N_Y$ but the boundary we know is a circle $e^{i\theta}u$ and so the symplectic area is the same as for the holomorphic disc $\{\tau u:\tau \in D\}$. Since $\Omega_{N_Y}$ is $S^1$-invariant and $H=\lambda$ on the boundary of this disc this means that it has symplectic area $\lambda$. 

Finally we will show that $(\omega+dd^c\psi_{\lambda})^{n-1}\neq 0$ on $S_{\lambda}^c\setminus Y$. Pick a complex hyperplane in the tangent space at a point $x\in S_{\lambda}^c\setminus Y$ not containing the tangents of the holomorphic disc going through $x$. This is then the projection of a complex subspace of the tangent space of $\mathcal{N}_Y$ at $(x,\tau),$ and we can choose this subspace to lie in the tangent space of $H^{-1}(\lambda)$. This subspace will not contain the tangents of the leaf going through $(x,\tau),$ and thus $\Omega_{can}^{n-1}$ is a volume form on it. Because of (\ref{envelopeequal}) this implies that $(\omega+dd^c\psi_{\lambda})^{n-1}\neq 0$.   
\end{proof}

\begin{remark}\label{rmk:tubularequilibrium}
  From the above we see that in fact the tubular neighbourhoods we have constructed in the previous section (or more precisely their image in $X$) are precisely the complement of the equilibrium set, that is  $\tilde{U}_{\lambda}=S_{\lambda}^c$ where $\tilde{U}_{\lambda}$ is as defined in \eqref{eq:deftildeU} (to see this combine \eqref{eq:psilambdaequalsalphalambda} and Proposition 5.3).   In fact the proof shows even more, namely that the pair $(S_{\lambda}^c\setminus
  Y,\omega+dd^c\psi_{\lambda})$ is canonically identified with the
  (semi)symplectic quotient $H^{-1}(\lambda)/S^1$ of $(V_{can}\setminus
  \mathcal{Y},\Omega_{can})$.
\end{remark}

When $Y$ is just a point Theorem \ref{optimalregthm} reduces to the following local statement.

\begin{theorem}
Let $\phi$ be a smooth strictly plurisubharmonic function on the unit ball $B$ in $\mathbb{C}^n$. For small $\lambda$ (i.e. $\lambda<\Lambda_{can}$) the envelope $$\psi_{\lambda}:=\sup\{\psi\in PSH(B): \psi\leq \phi, \nu_0(\psi)\geq \lambda\}$$ has optimal regularity. Also, for any bounded holomorphic function $f$ on $$B_{\lambda}:=S_{\lambda}^c$$ we have that $$\frac{1}{\textrm{vol}(B_{\lambda})}\int_{B_{\lambda}}fdV_{\phi}=f(0),$$ where the $dV_{\phi}$ denotes the Monge-Ampere measure $(dd^c\phi)^n/n!$
\end{theorem}

\begin{proof}
It immediately follows from Theorem \ref{optimalregthm2} that $\psi_{\lambda}$ has optimal regularity as long as $\lambda$ is small ($\lambda<\Lambda_{can}$).

Now let $f$ be a bounded holomorphic function on $B_{\lambda}$. Also let $T:U\to B$ be the canonical tubular neighbourhood of $0$ in $(B,dd^c\phi),$ thus $U$ is an $S^1$-invariant neighbourhood of the origin in $T_0 B\cong \mathbb{C}^n,$ and $\tilde{\omega}:=T^*dd^c\phi$ is K\"ahler and invariant. We also note that $T^{-1}(B_{\lambda})=H^{-1}([0,\lambda))\cap U$ and so this is also invariant. Now pulling back the integral we get that
\begin{eqnarray*}
\int_{B_{\lambda}}fdV_{\phi}=\int_{\{H\leq \lambda\}}f( T(z))\tilde{\omega}^n=\int_{\tau\in S^1}(\int_{\{H\leq \lambda\}}f(T(\tau z)\tilde{\omega}^n)d\tau= \\ =\int_{\{H\leq \lambda\}}(\int_{\tau\in S^1}f(T(\tau z))d\tau)\tilde{\omega}^n.
\end{eqnarray*}
Since the map $\tau \mapsto T(\tau z)$ extends to a holomorphic disc in $B$ passing through the origin at its centre it follows from the mean value property that $$\int_{\tau\in S^1}f(T(\tau z))d\tau=f(0),$$ at which point we are done.
\end{proof} 

\section{Volume of Tubular Neighbourhoods}

\begin{theorem}\label{thm:volumegrowthrepeat}
Assume $X$ is compact, and let $p\colon \tilde{X}\to X$ denote the blowup of $X$ along $Y$ with exceptional divisor $E$.  Then the volume of the tubular neighbourhood $T_{\lambda}$ is
$$Vol_{\lambda}:= \int_{\tilde{U}_{\lambda}} \frac{\omega^n}{n!} = \frac{1}{n!}\left(\int_X [\omega]^{n} - \int_{\tilde{X}} (p^*[\omega] -\lambda [E])^n\right)$$
where $[E]$ denotes the class of the current of integration along $E$.  
\end{theorem}
\begin{proof}
Let $\omega_E$ be a smooth closed form cohomologuous to $[E]$ and $\alpha_E$ be an $\omega_E$-psh function such that $dd^c\alpha_E+\omega_E=[E]$. If $z_i$ are local coordinates and $E$ is locally defined by the equation $z_1=0$ then $$\alpha_E=\ln|z_1|+O(1).$$ 

The fact that $\psi_{\lambda}=\alpha_{\lambda}$, see \eqref{eq:psilambdaequalsalphalambda}, combined with Proposition \ref{prop:alphalambda} and Remark \ref{rmk:alphagrowth} shows that $p^*\psi_{\lambda}=\lambda \ln|z_1|+O(1)$ locally. It thus follows that $p^*\psi_{\lambda}-\lambda\alpha_E$ is a locally bounded $p^*\omega-\lambda \omega_E$-psh function. Thus the class $p^*[\omega]-\lambda [E]$ is nef and $$\int_{\tilde{X}} (p^*[\omega] -\lambda [E])^n=\int_{\tilde{X}}(p^*\omega-\lambda \omega_E+dd^c(p^*\psi_{\lambda}-\lambda\alpha_E))^n.$$ Here the integrand of right hand side is meant to interpreted in the weak sense as a Monge-Amp\`ere measure. It is easy to see that $p^*\psi_{\lambda}-\lambda\alpha_E$ can be characterized as an extremal envelope, namely $$p^*\psi_{\lambda}-\lambda\alpha_E=\sup\{\psi\leq -\lambda\alpha_E: \psi\in PSH(\tilde{X},p^*\omega-\lambda \omega_E)\}.$$ It then follows from  \cite[Corollary 2.5]{BermanDemailly} by Berman-Demailly that $$\int_{\tilde{X}}(p^*\omega-\lambda \omega_E+dd^c(p^*\psi_{\lambda}-\lambda\alpha_E))^n=\int_{p^{-1}(S_{\lambda})}(p^*\omega-\lambda (\omega_E+dd^c\alpha_E))^n,$$ where $S_{\lambda}:=\psi_{\lambda}^{-1}(0)$ and thus $p^{-1}(S_{\lambda})=\{p^*\psi_{\lambda}-\lambda \alpha_E=-\lambda \alpha_E\}$. Clearly $\omega_E+dd^c\alpha_E$ vanishes outside of $E$ and hence on $p^{-1}(S_{\lambda})$ so we get that $$\int_{\tilde{X}} (p^*[\omega] -\lambda [E])^n=\int_{p^{-1}(S_{\lambda})}p^*\omega^n=\int_{S_{\lambda}}\omega^n.$$ Finally since $\tilde{U}_{\lambda}=S_{\lambda}^c$ (see Remark \ref{rmk:tubularequilibrium}) the Theorem follows.
\end{proof}

\noindent {\sc Julius Ross,  Department of Pure Mathematics and Mathematical Statistics, University of Cambridge, UK. \\j.ross@dpmms.cam.ac.uk}\vspace{2mm}\\ 

\noindent{\sc David Witt Nystr\"om, Department of Pure Mathematics and Mathematical Statistics,  University of Cambridge, UK. \\\quad d.wittnystrom@dpmms.cam.ac.uk, danspolitik@gmail.com}

\end{document}